\documentclass[12pt,reqno]{article}
\usepackage{geometry}
\usepackage[noblocks]{authblk}
\usepackage[bitstream-charter]{mathdesign}%font
\usepackage[T1]{fontenc}%mathematical papers in Elsevier journals
\usepackage{url} 
\usepackage{amsmath,amsthm,mathrsfs,lineno}
\usepackage{latexsym}

\usepackage{textcomp}
\usepackage{enumitem}

\usepackage{graphicx,booktabs,multirow}
\usepackage{tikz}
\usepackage{appendix}
\usepackage{yhmath}
\usepackage{float}
\usepackage{color}
\usepackage{subfigure}
\usepackage{rotating}
\usepackage[section]{placeins}%figure in one section
\usepackage{setspace}
\usepackage{pdfpages}

\newtheorem{theorem}{Theorem}
\newtheorem{lemma}{Lemma}
\newtheorem{corollary}{Corollary}
\newtheorem{observation}{Observation}
\newtheorem{problem}{Problem}
\newtheorem{proposition}{Proposition}
\newtheorem{remark}{Remark}
\newtheorem{claim}{Claim}

\newtheorem{conjecture}{Conjecture}
\newtheorem{definition}{Definition}
\newtheorem{case}{Case}

\newcommand\red[1] {{\color{red} #1}}

\usepackage[colorlinks,
linkcolor=red!60!black,
anchorcolor=blue,
citecolor=blue!60!black
]{hyperref}
\usepackage{url}

\numberwithin{equation}{section}

\usepackage{pdfpages} %!PDFLATEX
\usepackage{caption}

\tikzstyle{none}=[inner sep=0mm]
\tikzstyle{bluenode}=[fill=blue, draw=black, shape=circle, minimum size=0.2cm, 
inner sep=0.2pt]
\tikzstyle{blacknode}=[fill=black, draw=black, shape=circle, minimum 
size=0.2cm, inner sep=0.2pt]
\tikzstyle{rednode}=[fill={rgb,255: red,244; green,0; blue,0}, draw=black, 
shape=circle, minimum size=0.2cm, inner sep=0.2pt]
\tikzstyle{square}=[draw=black, fill=yellow, shape=rectangle, minimum 
size=0.3cm, inner sep=0.8pt]
\tikzstyle{whitenode}=[fill={rgb,255: red,245; green,245; blue,245}, draw=black, shape=circle, minimum size=0.1cm, inner sep=0.5pt]
\tikzstyle{blackedge}=[-, draw=black, fill=none, line width=0.15mm]
\tikzstyle{blueedge}=[-, draw=blue, fill=none, line width=0.3mm]
\tikzstyle{blackedge_thick}=[-, draw=black, line width=0.45mm, fill=none]
\tikzstyle{black_thick}=[-, draw=black, line width=0.5mm, fill=none]
\tikzstyle{rededge}=[-,  line width=0.3mm,draw=red]
\tikzstyle{rededge_thick}=[-, line width=0.45mm, draw=red]
\tikzstyle{black_thick_opacity}=[-, -, draw={rgb,255: red,91; green,87; blue,84}, 
fill=none, line width=0.4mm, opacity=0.7]
\tikzstyle{blue_thick}=[-, line width=0.5mm, draw=blue]
\newcommand{\proofend}{{\hfill$\Box$}}

\usetikzlibrary{backgrounds}
\usetikzlibrary{arrows}
\usetikzlibrary{shapes,shapes.geometric,shapes.misc}

\let\oldbibliography\thebibliography
\renewcommand{\thebibliography}[1]{%
  \oldbibliography{#1}%
  \setlength{\itemsep}{-8pt}%
  \setlength{\baselineskip}{8pt}
  \setlength{\lineskiplimit}{-\maxdimen}
}

\pgfdeclarelayer{edgelayer}
\pgfdeclarelayer{nodelayer}
\pgfsetlayers{background,edgelayer,nodelayer,main} 
\begin{document}

 \captionsetup[figure]{labelfont={bf},name={Fig.},labelsep=period}
	%\linenumbers 
	\title{4-connected 1-planar  chordal graphs are Hamiltonian-connected
	}
{\small
\author[1]{Licheng Zhang,\thanks{Email: lczhangmath@hnu.edu.cn}}
\author[2]{Yuanqiu Huang,\thanks{Corresponding author. Email: hyqq@hunnu.edu.cn.}}
\author[3]{Shengxiang Lv,\thanks{Email:lvsxx23@126.com}}
\author[4]{Fengming Dong\thanks{Email:fengming.dong@nie.edu.sg}}}

\affil[1]{\small School of Mathematics, Hunan  University, China}
\affil[2]{\small College of Mathematics and Statistics, Hunan Normal University, China}
\affil[3]{\small School of Mathematics and Statistics,Hunan University of Finance and Economics, China}
\affil[4]{\small National Institute of Education, Nanyang Technological University, Singapore}

%	\date{\today}
\date{}
\maketitle
	
\begin{abstract}
 Tutte proved that 4-connected planar graphs are Hamiltonian. 
It is unknown if there is an analogous result on $1$-planar graphs. 
 % A similar problem on 1-planar graphs remains unsolved: Is every 6-connected (or 7-connected) 1-planar graph  Hamiltonian?  
 In this paper,  we characterize 
 4-connected 1-planar chordal graphs,
 and show that all such graphs are  Hamiltonian-connected.
 A crucial tool used in our proof is a characteristic of 1-planar 4-trees. 
 %In addition, we  obtain some properties on 1-planar chordal graphs



\end{abstract}

%\tableofcontents

\section{Introduction}

Only finite simple connected graphs are considered in this paper.  
For a graph $G$,  let $V(G)$ and   $E(G)$ 
denote its vertex set and edge set.
For any subset $V_0$ of $V(G)$, 
the  subgraph of $G$ induced by $V_0$, 
denoted by $G[V_0]$, is the graph
with vertex set $V_0$ and 
edge set $\{uv\in E(G): u,v\in V(G)\}$.
Let $G-V_0$ denote the subgraph of $G$ induced by $V(G)\setminus V_0$
when $V_0\ne V(G)$.
A subset $S$ of  $V(G)$ is called 
a  \textit{separator} if  $G-S$ is disconnected, and a separator $S$ is 
called a {\it $k$-separator} if $|S|=k$.
For a non-complete graph $G$,
its {\it connectivity}, denoted by 
$\kappa(G)$, is defined to be 
the minimum value of $|S|$
over all separators $S$ of $G$,
and $\kappa(G)=|V(G)|-1$
if $G$ is connected.

A \textit{Hamiltonian path (resp. cycle) } in $G$ 
is a path (resp. cycle) in $G$ that visits every vertex of $G$ exactly
once.  A graph $G$ is called  \textit{Hamiltonian} if $G$ contains a Hamiltonian cycle. 
A graph $G$ is \textit{Hamiltonian-connected} if every two vertices in $G$ are connected by a Hamiltonian path. 
Clearly, all  Hamiltonian-connected graphs of orders at least three 
are Hamiltonian. 

A \textit{drawing} of a graph $G=(V,E)$ is a mapping $D(G)$ that assigns to each vertex in $V$ a distinct point in the plane  and to each edge $uv$ in $E$ a continuous arc connecting $D(u)$ and $D(v)$.  A drawing
is said to be \textit{good} if no edge crosses itself, no two edges cross more than once, and no two edges incident 
with a vertex cross each other. 
All drawings considered in this paper are good ones.
A graph is called {\it planar} if has a drawing so that no edges cross each other. 
The problem of Hamiltonicity of planar graphs  has received considerable attention. Whitney \cite{Whitney} showed that every 4-connected plane triangulation is Hamiltonian, and Tutte \cite{Tutte} extended
this conclusion to every 4-connected planar graph. Tutte's result has been strengthened in various ways. For example, Thomassen \cite{Thomassen} proved that every 4-connected planar graph is Hamiltonian-connected, 
and Thomas and Yu \cite{Thomas} proved that 4-connected projective-planar graphs are Hamiltonian. 
For more references, the reader 
may refer to  \cite{Ellingham,Lai,Ozeki2021,Ozeki}.

A graph is  called \textit{$1$-planar} if has a drawing so that any edge 
is crossed at most once, and   such a drawing is also called a 
\textit{$1$-plane graph}.
1-planar graphs are introduced by Ringel (1965) \cite{Ringel} in the connection with the problem of the simultaneous colouring of the vertices and faces of plane graphs.  
Since then, many properties of 1-planar graphs have been widely investigated (see \cite{Biedl2022, Brandenburg, Fabrici, Gollin,Liu,Suzuki} for example).  
It is known that 1-planar graphs have minimum degree at most 7,  and thus 
they are at most $7$-connected \cite{Fabrici2007}.
People try to establish  a foundational result on Hamiltonicity of highly connected 1-planar graphs, akin to Tutte's work on 4-connected planar graphs. A 1-planar graph $G$ with $|V(G)|\ge 3$ vertices has at most 
$4|V(G)|-8$ edges \cite{Bodendiek}. A 1-planar graph  is called {\it optimal} if  $|E(G)|=4|V(G)|-8$.  
Fujisawa et al. \cite{Fujisawa,Suzuki} showed that every optimal 1-planar graphs has  connectivity  4 or 6.   Hud\'{a}k et al.  \cite{Hudak}  proved that every optimal 1-planar is Hamiltonian.    Fabrici et al. \cite{Fabrici} proved that a 3-connected locally maximal 1-planar graph $G$ is Hamiltonian if it has at most three 3-vertex-cuts,
implying that 
%In particular, the authors  \cite{Fabrici} also showed that
all 4-connected maximal 1-planar graphs are Hamiltonian.
However, there are still many  non-Hamiltonian 1-planar graphs with connectivity 4 or 5 \cite{Biedlpre,Fabrici}.  
Actually, it is even unknown if 1-planar graphs with connectivity 6 or 7
are Hamiltonian.

\begin{problem}[\cite{Biedlpre,Fabrici}]
	Is every $6$-connected (or $7$-connected) $1$-planar graph is  Hamiltonian?
\end{problem}

A \textit{hole} is an induced cycle of length at least four. 
A graph is \textit{chordal} if it does not contain holes.  
In this article, we will establish the following conclusion
on the Hamiltonicity of $1$-planar graphs.

\begin{theorem}\label{4_conn}
	Let $G$ be a $1$-planar  graph. If $G$ is $4$-connected  and choral, then $G$ is Hamiltonian-connected.
\end{theorem}

Note that the conclusion in Theorem~\ref{4_conn} does not hold 
if some condition is replaced by a weaker one. 
Some details are given in following remark. 

		\renewcommand{\theenumi}{\rm (\roman{enumi})}

\begin{remark}\label{r1}
	\begin{enumerate}

		\item There are non-Hamiltonian $4$-connected $1$-planar graphs (see 
		\cite{Hudak}); 
		
		\item It is easy to construct $4$-connected chordal graphs that are not Hamiltonian. For example, 
		the graph obtained from 
		a graph $H$ isomorphic to $K_4$ 
		by adding $r$ vertices $v_1, \dots, v_r$, where $r\ge 5$,  
		and adding edges joining 
		 each $v_i$ to every vertex in $H$ is not Hamiltonian. 
		
		\item 
	There exist $3$-connected  planar chordal graphs  that are not Hamiltonian \cite{Nishizeki}.
	We can also construct $3$-connected 
	$1$-planar chordal graphs that are neither planar nor Hamiltonian.
	 An example is shown in 
	  Fig. \ref{NH_ch1_planargraphs}.
	  It is easy to verify that this graph $G_0$ is
	  chordal, $1$-planar and $3$-connected, but  not $4$-connected. This graph is not 
	  Hamiltonian due to the fact that 
	  the subgraph of $G_0$ obtained by removing 
	  vertices $0,1, 3,5$ and $7$ has six components. 
	  Taking any one of such $1$-planar graphs with a $3$-region $f$ whose boundary edges are not crossed, 
	  we can obtained a new one 
	  by identifying the three vertices 
	  of $f$ with vertices $0,11$ and $12$ 
	  in $G_0$. 
Hence there are infinity many $3$-connected 
$1$-planar chordal graphs that are neither planar nor Hamiltonian.
	\iffalse 
	 examples as follows: $G$ is the graph with a 1-planar drawing as shown in  It is easy to check that $G$ is chordal, $\kappa(G)=3$  and  is non-planar since $G[\{0,1,2,11,12\}]\cong K_5$. Let $S=\{0,1,2,5,7\}$. Note that $G-S$ has six connected-components. So it is easy to see that $G$ is not Hamiltonian. Next, we can replace the face $[5675]$ by $G$, and the resulting  1-planar graph also has  connectivity $3$ and  chordal, non-Hamiltonian (we still choose $S$ and use same reason). The above process can be recursively 
	continued.\fi 
	\end{enumerate} 
\end{remark}

 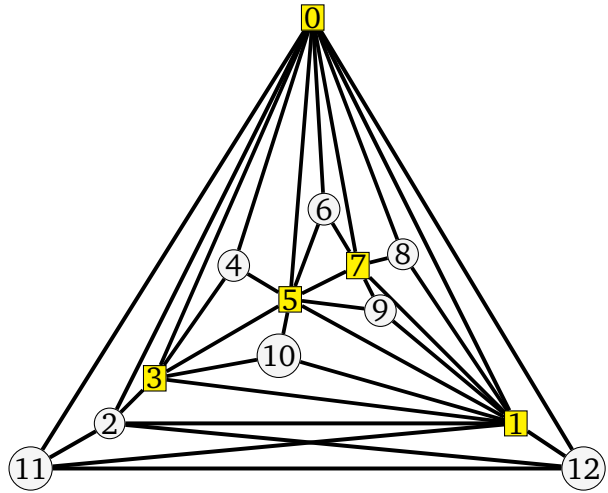
\begin{figure}[h!]
         \centering
\begin{tikzpicture}[scale=0.6]
  \tikzset{pureedge/.style = {line width=1.5pt, black}}
	\begin{pgfonlayer}{nodelayer}
		\node [style=square] (0) at (0, 10) {$0$};
		\node [style=whitenode] (1) at (-4.5, 1) {$2$};
		\node [style=square] (2) at (4.5, 1) {1};
		\node [style=square] (3) at (-3.5, 2) {3};
		\node [style=whitenode] (4) at (0.25, 5.75) {6};
		\node [style=whitenode] (5) at (2, 4.75) {8};
		\node [style=square] (6) at (1, 4.5) {7};
		\node [style=whitenode] (7) at (1.5, 3.5) {9};
		\node [style=whitenode] (8) at (-1.75, 4.5) {4};
		\node [style=square] (9) at (-0.5, 3.75) {5};
		\node [style=whitenode] (10) at (-0.75, 2.5) {10};
		\node [style=whitenode] (11) at (-6.25, 0) {11};
		\node [style=whitenode] (12) at (6, 0) {12};
	\end{pgfonlayer}
	\begin{pgfonlayer}{edgelayer}
		\draw [style=pureedge] (0) to (1);
		\draw [style=pureedge] (1) to (2);
		\draw [style=pureedge] (2) to (0);
		\draw [style={black_thick}] (0) to (4);
		\draw [style={black_thick}] (4) to (6);
		\draw [style={black_thick}] (6) to (5);
		\draw [style={black_thick}] (5) to (0);
		\draw [style={black_thick}] (0) to (6);
		\draw [style={black_thick}] (8) to (3);
		\draw [style={black_thick}] (3) to (1);
		\draw [style={black_thick}] (3) to (10);
		\draw [style={black_thick}] (10) to (2);
		\draw [style={black_thick}] (2) to (3);
		\draw [style={black_thick}] (8) to (9);
		\draw [style={black_thick}] (0) to (8);
		\draw [style={black_thick}] (0) to (3);
		\draw [style={black_thick}] (9) to (10);
		\draw [style={black_thick}] (3) to (9);
		\draw [style={black_thick}] (9) to (4);
		\draw [style={black_thick}] (0) to (9);
		\draw [style={black_thick}] (9) to (7);
		\draw [style={black_thick}] (2) to (7);
		\draw [style={black_thick}] (9) to (2);
		\draw [style={black_thick}] (6) to (7);
		\draw [style={black_thick}] (6) to (2);
		\draw [style={black_thick}] (2) to (5);
		\draw [style=pureedge] (0) to (11);
		\draw [style=pureedge] (11) to (1);
		\draw [style=pureedge] (12) to (11);
		\draw [style=pureedge] (11) to (2);
		\draw [style=pureedge] (1) to (12);
		\draw [style=pureedge] (12) to (2);
		\draw [style=pureedge] (12) to (0);
		\draw [style={black_thick}] (9) to (6);
	\end{pgfonlayer}
\end{tikzpicture}

      \caption{A chordal  and 1-planar 
      	graph $G_0$ 
      	that has connectivity $3$ and is not  Hamiltonian.}
      
        \label{NH_ch1_planargraphs}
    \end{figure}

\iffalse 
 The toughness of a graph is  closely associated with Hamiltonicity. 
 The \textit{toughness} of a graph $G$,
 denoted by $\tau(G)$, 
 is the minimum value of $\frac{|X|}{c(G-X)}$ over all  
 non-empty subsets $X$ of $V(G)$
 with $c(G-X)>1$, where  
 $c(H)$ is the number of components of a graph $H$.
  The toughness of a complete graph is defined as being infinite. 
  We say that a graph is \textit{$t$-tough} if its toughness is at least $t$. 
  There are some known results on 
  the Hamiltonicity of chordal graphs 
  in terms of their toughness. 
  For example, every 10-tough chordal graph is Hamiltonian \cite{Kabela},
  and  every %\red{more than 1-tough}
  chordal planar graph of order 
  at least three and toughness greater than one is Hamiltonian \cite{Bohme}.

  % Every \red{more than 1-tough} chordal planar graph \red{with at least three vertices of toughness greater than one} is Hamiltonian \cite{Bohme}.
   
   Note that a $t$-tough graph is always $\lceil 2t \rceil $-vertex-connected \cite{Chvatal}. So the corollary below follows directly from Theorem \ref{4_conn}.
   
   \begin{corollary}\label{corol1}
   	Every more than $\frac{3}{2}$-tough chordal $1$-planar graph  is Hamiltonian-connected.
   \end{corollary}
   
\fi

The remaining sections of this paper are organized as follows. In Section 2, we explain some terminology, notations and some preliminary lemmas, and in Section 3, we give some fundamental results on chordal graphs. In Section 4, we introduce a 4-join operation for generating 1-planar 4-trees. 
In Section 5, we characterize 
$4$-connected and $1$-planar chordal graphs. 
In Section 6, we prove Theorem~\ref{4_conn} by the result 
obtained in Section 5. 
Finally we leave some interesting open questions.

\section{Preliminaries}

%If the reader finds something unfamiliar, see a textbook of graph theory.
In this section, we introduce some notations and definitions.   
\iffalse 
Let $G=(V, E)$ be a connected graph. Let $v$ be a vertex of $G$. Let $N_G(v)$ denote the set of neighbors of $v$ in $G$, and the {\it degree} of
$v$, denoted by $d_G(v)$, is the number of edges in $G$ incident with $v$. $G^{\prime}=\left(V^{\prime}, E^{\prime}\right)$ is a \textit{subgraph} of $G$ if and only if $V^{\prime} \subseteq V$, and $E^{\prime} \subseteq E$.  A graph $G''=\left(V'', E''\right)$ is an \textit{induced subgraph} of $G$ if $V'' \subseteq V$ and all edges of $G$ having both ends in $V''$ form edge set $E''$.  We denote the subgraph induced by $V^{\prime} \subseteq V$ by $G\left[V^{\prime}\right]$.  For an edge $e$, denote $G\setminus e$ by the subgraph obtained by deleting $e$ from $G$. For a subset $S$ of $V$, the graph $G-X$ is the induced subgraph $G[V-S]$ of $G$. We denote by $G\cong H$ that graphs $G$ and $H$ are isomorphic.  
\fi

Let $G$ be a graph with  a drawing $D$.
For a  subgraph $H$ of graph $G$,
the subdrawing  of $D$ induced by $H$ is called a \textit{restricted drawing} of $D$, denoted by $D|H$. 
Two  drawings of a graph are \textit{isomorphic} if there is a  homeomorphism of the sphere that maps one drawing to the other. 
If $D$ is a 1-planar drawing, 
 an edge of $D$ is \textit{crossed} if it crosses another edge, and is \textit{uncrossed} otherwise.   
 A \textit{face} in a 1-planar drawing $D$ is given by a cyclic list of edges
and edge segments, where the latter occurs in the case of a crossing. The \textit{planar skeleton} $\mathcal{S}(D)$ of $D$ is the 
subgraph of $G$ by removing all crossed edges of $D$.  
For a face $f$ of $D$, 
denote $\partial( f )$ by  the  boundary of  $f$.  
We call $f$ is \textit{crossed} if $\partial( f)$ has a crossing, and is \textit{uncrossed} otherwise.
All ertices and crossings on $\partial(f)$  are called \textit{ corners}. A face $f$ is triangular if $\partial(f)$ has exactly three corners. 
A \textit{triangulated} 1-plane graph is one in which all faces are triangular.
%For any drawing $D$, let $\mathrm{cr}_D$ denote the number of crossings in $D$. The \textit{crossing number} of a graph $G$, denoted by $\mathrm{cr}(G)$, is the minimum value of $\mathrm{cr}_D$ over all
%drawings $D$ of $G$.

 For other terminology and notations not defined here, we refer to \cite{Bondy}, and the following are some auxiliary lemmas in this paper.

%\begin{lemma}[\cite{suzuki}] \label{kn1_planar}
%let $g$ be a graph with  at most six vertices. then $g$ is 1-planar. 
%\end{lemma}

Some fundamental results on $1$-planar graphs are provided below.

\begin{lemma}[\cite{Suzuki}]
	\label{K4n-8}
	For any  $1$-planar graph $G$ 
	of order $n\ge 3$, 
$|E(G)|\le 4n-8$, and moreover, $|E(G)|\le 4n-9$ if $n=7$ or $9$.
\end{lemma}

\begin{lemma}[\cite{Suzuki}]\label{K6}
The complete graphs $K_5$ and $K_6$ have exactly one (up to isomorphism) 1-planar drawings as shown in Fig. \ref{fig:K6-drawings}, respectively.
\end{lemma}

 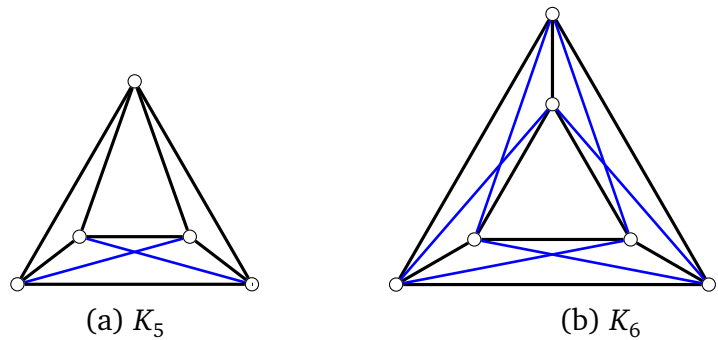
\begin{figure}[h!]
        \centering
        \begin{tikzpicture}[scale=0.6]
            \tikzset{vertex/.style = {circle, draw, fill=white, inner 
            sep=0pt, minimum width=5pt}}
            \tikzset{edge0/.style = {line width=1.2pt, black}}
            \tikzset{edge1/.style = {line width=1pt, blue, opacity=0.7}}
            
            \node [vertex] (v0) at (210:3) {};
            \node [vertex] (v1) at (330:3) {};
            %\node [vertex] (v2) at (90:3) {};
            \node [vertex] (w0) at (200:1.3) {};
            \node [vertex] (w1) at (340:1.3) {};
            \node [vertex] (w2) at (90:3) {};
            \draw [edge1] (v1) edge (w0) {};
            \draw [edge0] (v1) edge (w1) {};
            \draw [edge0] (v1) edge (w2) {};
            \draw [edge0] (v1) edge (v1) {};
            \draw [edge1] (v0) edge (w1) {};
            \draw [edge0] (v0) edge (v1) {};
            \draw [edge0] (v0) edge (w2) {};
            \draw [edge0] (w0) -- (w1) -- (w2) -- (w0){};
            \draw [edge0] (v0) -- (w0);
            
        \end{tikzpicture}
    \hspace{1.5 cm}
      \begin{tikzpicture}[scale=0.8]
                  \tikzset{vertex/.style = {circle, draw, fill=white, inner 
                  sep=0pt, minimum width=5pt}}
                  \tikzset{edge0/.style = {line width=1.2pt, black}}
                  \tikzset{edge1/.style = {line width=1pt, blue, opacity=0.7}}
                  
                  \node [vertex] (v0) at (210:3) {};
                  \node [vertex] (v1) at (330:3) {};
                  \node [vertex] (v2) at (90:3) {};
                  \node [vertex] (w0) at (210:1.5) {};
                  \node [vertex] (w1) at (330:1.5) {};
                  \node [vertex] (w2) at (90:1.5) {};
                  
                  \foreach \i in {0,1,2} {
                      \pgfmathtruncatemacro{\j}{Mod(\i+1,3)}
                      \foreach \v in {v,w} {
                          \draw [edge0] (\v\i) edge (\v\j) {};
                      };
                      \draw [edge0] (v\i) edge (w\i) {};
                      \draw [edge1] (v\i) edge (w\j) {};
                      \draw [edge1] (v\j) edge (w\i) {};
                  };
              \end{tikzpicture}
          
          (a) $K_5$ \hspace{5 cm} (b)
          $K_6$
          
        \caption{The unique 1-planar drawings of  $K_5$ and $K_6$.}
      
        \label{fig:K6-drawings}
    \end{figure}

%
%\begin{lemma}
%Let $G$ be a $4$-connected graph with a $1$-planar drawing  $D$.
%Let $C=[uxvu]$ be a 3-cycle in $D^*$  where $x$ is a crossing. Then there 
%are no vertices inside $D|C$, and $uv$ is a uncrossed edge.
%\end{lemma}

\begin{lemma}\label{3cutface}
Let $D$ be a 1-planar drawing of $G$,
and let $H$ be a subgraph of $G$. 
If some vertex of $G-H$ 
lies inside a crossed triangular
 face $f$ of $D|H$, 
then $\kappa(G)\le 3$.
\end{lemma}

\begin{proof}
Let $f=[abca]$. Assume that $a$ is the unique crossing in $\partial(f)$. Let $v$ be a vertex lying inside $f$. If $bc$ is uncrossed, then $\{b,c\}$ is a 2-separator of $G$. If  $bc$ is crossed by $xy$, 
where $x\notin \partial(f)$, then $\{b,c,x\}$ is 
a 3-separator of $G$.  In either case, $ \kappa(G)\le 3$.
\end{proof}

\newcommand \OPF[2]
{
\widehat{#1}_{#2}
%#1(#2 \Box) %\textlbrackdbl)
}

If $e$ is an uncrossed edge 
on $\partial(f)$ for some 
face $f$ in a $1$-planar drawing $D$, 
let $\OPF{f}{e}$ denote the other 
face in $D$ such that $e$ is also 
on  $\partial(\OPF{f}{e})$.

\begin{lemma}\label{insert-v}
Let $D$ be a $1$-planar drawing of a $1$-planar graph $G$. 
Assume that vertex $v\in V(G)$  lies within face $f$ of $D':=D|(G-v)$. 
Then, the following conclusions hold:
\begin {enumerate}
\item 
the sum of the number of uncrossed corners 
(i.e., vertices in $G-v$) on $\partial(f)$
and the number of  uncrossed edges 
$e$ on $\partial(f)$
for which $\partial(\OPF{f}{e})$ has at least three uncrossed corners 
is at least $d_G(v)$; and 
\item if $D'$ is triangulated and $d_D(v)=4$, 
then $f$ is uncrossed and 
$\OPF{f}{e}$ is also uncrossed for 
at least one edge $e$ on $\partial(f)$. 
\end{enumerate} 
\end{lemma}

\begin{proof}
The result follows directly by 
the definition of $1$-planar
drawings. 
\end{proof}

A \textit{clique} of a graph $G$ is a subset $S$ of $V(G)$ such that 
$G[S]$ is complete. 
A \textit{$k$-clique} is a clique with 
exactly $k$ vertices.

\begin{lemma}\label{K6conn}
	For a $1$-planar graph $G$ 
	of order $n\ge 7$, 
	if $G$ contains a $6$-clique, then $\kappa(G)\le 3$.
\end{lemma} 
\begin{proof}
Let $D$ be a 1-planar drawing of $G$
and let  $S$ be a $6$-clique of $G$.  
By Lemma \ref{K6}, $D|{G[S]}$ is unique, which is shown in Fig \ref{fig:K6-drawings} (b). 

Since $n\ge 7$, there is a vertex $u\in V(G)\setminus S$. 
Then $u$ lies in a triangular face $f$ in $D|{G[S]}$. 
If $f$ is crossed, the conclusion follows from Lemma \ref{3cutface}.  
If $f$ is uncrossed, 
the three vertices on $\partial (f)$ 
form a 3-separator of $G$. 
Hence $\kappa(G)\le 3$.
\end{proof}

%%%%%%%%% the following results are not applied.
\iffalse 
For a 1-plane graph $G$, 
a \textit{separating cycle} of $G$ is a cycle $C$ in $G$ such that the curve traced by $C$ in $G$ does not self-intersect and has at least one vertex strictly inside and strictly outside.

\begin{lemma}[\cite{Biedlpre}]
	\label{b1}
Let $G$ be a triangulated $1$-plane graph without uncrossed separating triangle, and
assume that $G$ contains at least $6$ vertices. Then $G$ contains a spanning planar subgraph  that is  triangulated and  has no separating triangle.
\end{lemma}

\begin{lemma}
	[\cite{Thomassen}]\label{t2}
Every 4-connected planar graph is Hamiltonian-connected.
\end{lemma}

\fi

\section{Some fundamental results on chordal graphs}

In this section, 
we introduce some basic results on chordal graphs.
By the definition of chordal graphs, 
the following result holds.
 
 \begin{lemma}[p51,\cite{Voloshin}]\label{hereditary}
 Let $G$ be  a chordal graph. Then every induced subgraph of $G$ is chordal.
 \end{lemma}

A vertex $u$ in $G$ is said to be  \textit{simplicial} if 
either $d(u)=0$ or $N_G(u)$ is a clique.
An ordering $(u_1 u_2 \cdots u_n)$ of the vertices in $G$ is a 
called \textit{perfect 
	elimination order}, if  $u_i$ is a 
simplicial vertex of 
$G-\{u_j: 1\le j<i\}$ 
for all $i=1,2,\cdots,n-1$.  
Rose \cite{Rose1970} proved that a graph is chordal if and only if it has a perfect elimination ordering. 
Fix an integer $k \ge 1$, the class of \textit{ $k$-trees} is defined recursively as follows.  
The complete graph $K_k$ is the smallest $k$-tree, and 
a graph $G$ of order $n$,
where $n\ge k+1$, is a {\it $k$-tree}
if $G$ has a simplicial vertex $u$ 
of degree $k$ and 
$G-\{u\}$ is a $k$-tree. 
%A $k$-tree with $n$ vertices ( $n \ge k+1$ ) can be constructed from a $k$-tree with $n-1$ vertices by adding a vertex adjacent to all vertices of one of its $k$-cliques.  
Clearly, from the construction of $k$-trees, they naturally contain  perfect elimination orderings, and so $k$-trees form a subclass of chordal graphs.

\begin{lemma}[\cite{Dirac,Rose1970}] \label{cliques}
Let $G$ be a graph. The following statements are equivalent:
\begin{enumerate}
\item $G$ is chordal; 
\item every minimal separator 
of $G$ is a clique; and 
\item $G$ has a perfect elimination ordering.
\end{enumerate}
\end{lemma}

\begin{lemma}\label{k6_e}
	Let $G$ be a 4-tree with $|V(G)|=6$. Then $G\cong K_6\setminus e$. Moreover, $G$  has exactly three non-isomorphic  
	$1$-planar drawings,  as 
	shown in Fig. \ref{K6_edrwings}.
\end{lemma}

\begin{proof}
	By the definition of $4$-trees, any 4-tree with five vertices is $K_5$, and then  $G\cong K_6\setminus e$. 
	Due to  Korzhik \cite{Korzhik}, 
	$G$  has exactly three 1-planar drawings $A_1, A_2$ and $A_3$ shown in Fig. \ref{K6_edrwings}.
\end{proof}

 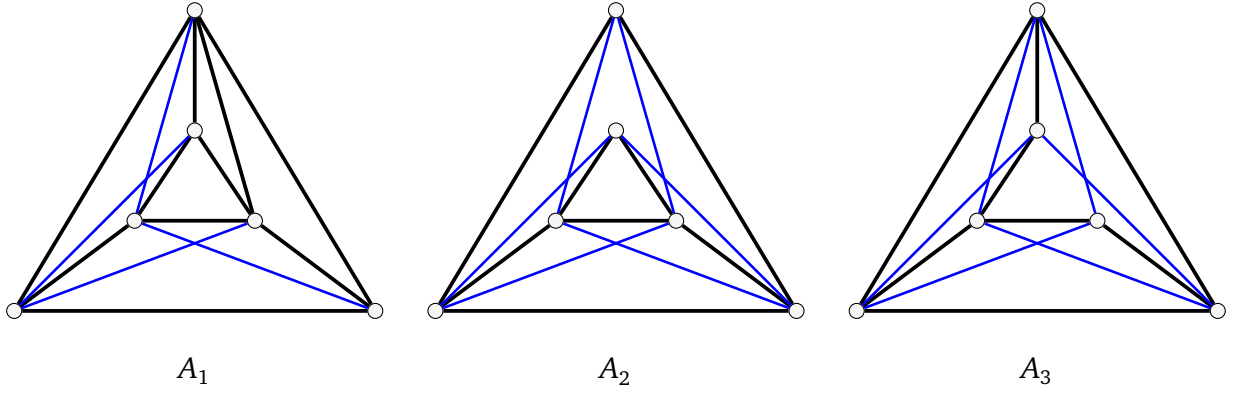
\begin{figure}[h!]
	\centering
	\begin{tikzpicture}[scale=0.8]
		\tikzset{blueedge_thick/.style = {line width=1pt, blue, opacity=0.7}}
\tikzset{whitenode/.style={fill={rgb,255: red,245; green,245; blue,245}, draw=black, shape=circle, minimum size=0.2cm, inner sep=0.1pt}}
		\node [style=whitenode] (0) at (-5, 6) {};
		\node [style=whitenode] (1) at (-8, 1) {};
		\node [style=whitenode] (2) at (-2, 1) {};
		\node [style=whitenode] (3) at (-4, 2.5) {};
		\node [style=whitenode] (4) at (-6, 2.5) {};
		\node [style=whitenode] (5) at (2, 6) {};
		\node [style=whitenode] (6) at (-1, 1) {};
		\node [style=whitenode] (7) at (5, 1) {};
		\node [style=whitenode] (8) at (3, 2.5) {};
		\node [style=whitenode] (9) at (1, 2.5) {};
		\node [style=whitenode] (10) at (9, 6) {};
		\node [style=whitenode] (11) at (6, 1) {};
		\node [style=whitenode] (12) at (12, 1) {};
		\node [style=whitenode] (13) at (10, 2.5) {};
		\node [style=whitenode] (14) at (8, 2.5) {};
		\node [style=whitenode] (15) at (-5, 4) {};
		\node [style=whitenode] (16) at (2, 4) {};
		\node [style=whitenode] (17) at (9, 4) {};
		\node [style=none] (18) at (-5, 0) {$A_1$};
		\node [style=none] (19) at (2, 0) {$A_2$};
		\node [style=none] (20) at (9, 0) {$A_3$};
		\draw [style={black_thick}] (0) to (1);
		\draw [style={black_thick}] (1) to (2);
		\draw [style={black_thick}] (2) to (0);
		\draw [style={black_thick}] (0) to (3);
		\draw [style={black_thick}] (3) to (2);
		\draw [style={black_thick}] (4) to (1);
		\draw [style={black_thick}] (4) to (3);
		\draw [style={black_thick}] (5) to (6);
		\draw [style={black_thick}] (6) to (7);
		\draw [style={black_thick}] (7) to (5);
		\draw [style={black_thick}] (8) to (7);
		\draw [style={black_thick}] (9) to (6);
		\draw [style={black_thick}] (9) to (8);
		\draw [style={black_thick}] (10) to (11);
		\draw [style={black_thick}] (11) to (12);
		\draw [style={black_thick}] (12) to (10);
		\draw [style={black_thick}] (13) to (12);
		\draw [style={black_thick}] (14) to (11);
		\draw [style={black_thick}] (14) to (13);
		\draw [style={black_thick}] (15) to (4);
		\draw [style={black_thick}] (15) to (3);
		\draw [style={black_thick}] (16) to (8);
		\draw [style={black_thick}] (17) to (14);
		\draw [style={black_thick}] (17) to (10);
		\draw [style={black_thick}] (0) to (15);
		\draw [style={black_thick}] (16) to (9);
		\draw [style={blueedge_thick}] (0) to (4);
		\draw [style={blueedge_thick}] (15) to (1);
		\draw [style={blueedge_thick}] (1) to (3);
		\draw [style={blueedge_thick}] (4) to (2);
		\draw [style={blueedge_thick}] (6) to (16);
		\draw [style={blueedge_thick}] (5) to (9);
		\draw [style={blueedge_thick}] (5) to (8);
		\draw [style={blueedge_thick}] (16) to (7);
		\draw [style={blueedge_thick}] (6) to (8);
		\draw [style={blueedge_thick}] (9) to (7);
		\draw [style={blueedge_thick}] (10) to (14);
		\draw [style={blueedge_thick}] (11) to (17);
		\draw [style={blueedge_thick}] (10) to (13);
		\draw [style={blueedge_thick}] (17) to (12);
		\draw [style={blueedge_thick}] (11) to (13);
		\draw [style={blueedge_thick}] (14) to (12);
	\end{tikzpicture}
	\caption{Three non-isomorphic 1-planar drawings of $K_6\setminus e$.}
	\label{K6_edrwings}
\end{figure}

\iffalse 
\begin{lemma}[\cite{Bohme}]
	\label{hereditary2}
Let $G$ be a $k$-connected chordal graph and $(u_1 u_2 \ldots u_n)$ be a 
perfect elimination order of $G$ and $i \in\{1, \ldots, n-1\}$. Then 
$G-\left\{u_1, \ldots, u_i\right\}$ is either $k$-connected or complete.
\end{lemma} 
\fi 

\iffalse 
\begin{lemma}[\cite{Rose1974}]
	\label{k_tree}
Let $G$ be a $k$-tree with $n$ vertices. Then
\begin{enumerate}
\item  $|E(G)| = kn-\frac{1}{2} k(k+1)$; 
and 
\item  if $n\ge k+1$, $\kappa(G)=k$.
\end{enumerate}
%In particular, if $G$ is a $4$-tree, $|E(G)| = 4n-10$.
\end{lemma}
\fi

\iffalse 
\begin{lemma}[\cite{Rose1974}]
	\label{ktrees_1}
 Let $G$ be a chordal graph with $n \geq k$ vertices. 
 If $G$ has  a $k$-clique but no $(k + 2)$-clique, then  $|E(G)| \leq kn-\frac{1}{2} k(k+1)$, 
 and if  $|E(G)| = kn-\frac{1}{2} k(k+1)$  then $G$ is a $k$-tree.
\end{lemma}
\fi

\begin{lemma}[\cite{Kabela19}]
	\label{prop1}
	Let $G$ be a $k$-connected chordal graph.  If $G$ has  no $(k + 2)$-clique, then $G$ is a $k$-tree.
\end{lemma}

The following lemma must have appeared somewhere, although we cannot find it. We provide a simple 
proof. 

\begin{lemma}
	\label{k-connectedchordal}
Let $G$ be a chordal graph with $n$ vertices. If $G$ is $k$-connected, then $|E(G)|\ge kn-\frac{1}{2}k(k+1)$.
\end{lemma} 

\begin{proof}
We can prove this result by induction 
on $n$.
Since $G$ is $k$-connected, 
we have $n\ge k+1$.
When $n=k+1$, $G$ is the complete 
graph $K_{k+1}$, and the result holds.
Now assume that $n\ge k+2$. 
By Lemma~\ref{cliques}, 
$G$ contains a simplicial vertex $u$. 
Since $G$ is $k$-connected, 
we have $d(u)\ge k$. 
Obviously, $G-\{u\}$ is $k$-connected  and chordal.
By induction,  the result holds 
for $G-\{u\}$, and thus 
$$
|E(G)|=|E(G-\{u\})|+d(u)
\ge k(n-1)-\frac{1}{2}k(k+1)+k
=kn-\frac{1}{2}k(k+1).
$$
Hence the result holds.
\end{proof}

\iffalse 
\begin{lemma}\label{prop1}
Let $G$ be a  chordal graph.  If $G$ is $k$-connected, and has  no $(k + 2)$-clique, then $G$ is a $k$-tree.
\end{lemma}

\begin{proof}\red{
Let $n$ be the order of $G$. 
We prove this result by induction on $n$.
Since $G$ is $k$-connected, 
we have $n\ge k+1$. 
If $n=k+1$, then $G$ is $K_{k+1}$, 
a $k$-tree. 

Now assume that 
$n\ge k+2$.
By Lemma~\ref{cliques}, 
$G$ contains a simplicial vertex $u$. 
Since $G$ is $k$-connected
and  has no $(k+2)$-clique, 
we have $d(u)=k$. 
Obviously, $G-\{u\}$ is a
$k$-connected chordal graph 
without $(k+2)$-clique.
By induction, $G-\{u\}$ is a $k$-tree.
Thus, $G$ is also a $k$-tree 
by the definition of $k$-trees.
}
\iffalse 
Let $n$ be the order of $G$. 
Since $G$ is $k$-connected, 
we have $n\ge k$. 
If $n=k$, $G$ is $K_k$, 
a $k$-tree. Now assume that 
$n\ge k+1$.

By  Lemma \ref{k-connectedchordal}, $|E(G)|\ge kn-\frac{1}{2}k(k+1)$. 
Since $G$ is $k$-connected
and $n\ge k+1$,
$G$ contains a $k$-clique
by Lemma \ref{clique2}. 
Then  by Lemma \ref{ktrees_1}, we have $|E(G)|\le kn-\frac{1}{2}k(k+1)$.
Thus, $|E(G)|= kn-\frac{1}{2}k(k+1)$. 
Finally, by  Lemma \ref{ktrees_1}, the lemma holds.\fi 
\end{proof}
\fi

\section{4-join operation}

In this section, we introduce
a graph operation, 
the 4-join operation,
which will be applied to generate 1-planar 4-trees.

Let $D$ be a $1$-plane graph. 
If $f_1:=[abv_1a]$ and $f_2:=[abv_2a]$ 
are uncrossed triangular faces in $D$
such that $\partial (f_1)$ and 
$\partial (f_2)$  share exactly 
one edge (i.e., $ab$)
and $D[\{a,b,v_1,v_2\}]\cong K_4$, 
then $f_1$ and $f_2$ are called 
\textbf{twin faces} in $D$,
as shown in Fig.~\ref{4_join} (a).

 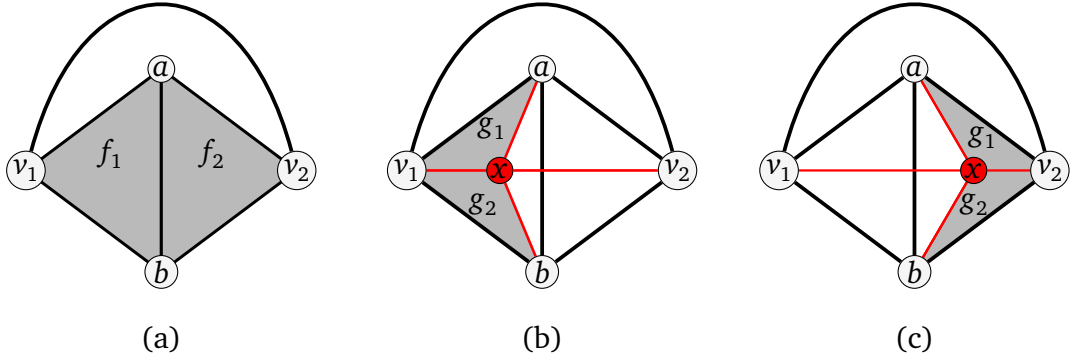
\begin{figure}[h!]
	\centering
	\begin{tikzpicture}[scale=0.45]
		\tikzset{rededge_thick/.style = {line width=1pt, red, opacity=0.7}}
		\tikzset{shadow/.style={fill={rgb,255: red,186; green,186; blue,186}, draw={black}, line width=0.4mm}}
		\tikzset{shadow_/.style={fill={rgb,255: red,186; green,186; blue,186}, draw={none}, line width=0.0mm}}
		\tikzstyle{black_thick}=[-, draw=black, line width=0.5mm, fill=none]
		\begin{pgfonlayer}{nodelayer}
			\node [style=whitenode] (0) at (-11.25, 3) {$a$};
			\node [style=whitenode] (1) at (-15.25, 0) {$v_1$};
			\node [style=whitenode] (2) at (-11.25, -3) {$b$};
			\node [style=whitenode] (3) at (-7.25, 0) {$v_2$};
			\node [style=none] (5) at (-12.75, 0.5) {$f_1$};
			\node [style=none] (6) at (-9.75, 0.5) {$f_2$};
			\node [style=whitenode] (7) at (0, 3) {$a$};
			\node [style=whitenode] (8) at (-4, 0) {$v_1$};
			\node [style=whitenode] (9) at (0, -3) {$b$};
			\node [style=whitenode] (10) at (4, 0) {$v_2$};
			\node [style=rednode] (11) at (-1.25, 0) {$x$};
			\node [style=none] (18) at (-1.5, 1.25) {$g_1$};
			\node [style=none] (19) at (-1.75, -1) {$g_2$};
			\node [style=none] (28) at (-11.25, -5) {(a)};
			\node [style=none] (29) at (0, -5) {(b)};
			\node [style=whitenode] (30) at (11, 3) {$a$};
			\node [style=whitenode] (31) at (7, 0) {$v_1$};
			\node [style=whitenode] (32) at (11, -3) {$b$};
			\node [style=whitenode] (33) at (15, 0) {$v_2$};
			\node [style=rednode] (36) at (12.75, 0) {$x$};
			\node [style=none] (37) at (11, -5) {(c)};
			\node [style=none] (38) at (13, 1) {$g_1$};
			\node [style=none] (39) at (12.75, -1) {$g_2$};
		\end{pgfonlayer}
		\begin{pgfonlayer}{edgelayer}
						\draw [style={black_thick}] (7) to (9);
			\draw [style={black_thick}, bend left=75, looseness=2.00] (8) to (10);
			\draw [style={black_thick}] (30) to (32);
			\draw [style=shadow] (2.center)
			to (0.center)
			to (1.center)
			to cycle;
			\draw [style=shadow] (2.center)
			to (0.center)
			to (3.center)
			to cycle;
			\draw [style={shadow_}] (11.center)
			to (7.center)
			to (8.center)
			to cycle;
			\draw [style={shadow_}] (8.center)
			to (11.center)
			to (9.center)
			to cycle;
			\draw [style={rededge_thick}] (7) to (11);
			\draw [style={rededge_thick}] (11) to (8);
			\draw [style={rededge_thick}] (11) to (9);
			\draw [style={rededge_thick}] (11) to (10);
			\draw [style={black_thick}, bend left=75, looseness=2.00] (31) to (33);
			\draw [style={black_thick}] (30) to (31);
			\draw [style={black_thick}] (31) to (32);
			\draw [style=rededge] (36) to (32);
			\draw [style={shadow_}] (32.center)
			to (33.center)
			to (30.center)
			to (36.center)
			to cycle;
			\draw [style=rededge] (30) to (36);
			\draw [style=rededge] (36) to (32);
			\draw [style=rededge] (32) to (33);
			\draw [style=rededge] (30) to (36);
			\draw [style=rededge] (36) to (32);
			\draw [style=rededge] (32) to (33);
			\draw [style={black_thick}] (33) to (30);
			\draw [style=rededge] (36) to (33);
			\draw [style=rededge] (36) to (31);

			\draw [style={black_thick}] (33) to (32);
			\draw [style={black_thick}, bend left=75, looseness=2.00] (1) to (3);
			\draw [style={black_thick}] (7) to (8);
			\draw [style={black_thick}] (8) to (9);

			\draw [style={black_thick}] (7) to (10);
			\draw [style={black_thick}] (10) to (9);
		\end{pgfonlayer}
	\end{tikzpicture}
	
	\caption{4-join operation on twin  faces $f_1$ and $f_2$}     
	\label{4_join}
\end{figure}

\begin{definition}
	[\textbf{4-join operation}]
	Let $f_1:=[abv_1a]$ and $f_2:=[abv_2a]$ be twin faces in 
	a plane graph $D$, as shown in 
	Fig.~\ref{4_join} (a). 
	A $4$-join operation
	at  the order pair $(f_1,f_2)$  is defined as follows:
	\begin{enumerate}
		\item put a new vertex $x$ within face $f_1$, and 
		\item  draw edges $xa$,$xb$, $xv_1$ inside $f_1$, and draw edge 
		$xv_{2}$ that crosses $ab$,
		as shown in Fig. \ref{4_join} (b).
	\end{enumerate}
	The  $1$-plane graph drawing obtained by the 4-join operation 
	at $( f_1,f_2)$
	is denoted by $D \oplus ( f_1,f_2)$, as shown at Fig.~\ref{4_join}  (b).\footnote{Note that 
	$D \oplus ( f_2,f_1)$, shown at Fig.~\ref{4_join} (c), 
	is different from 
	$D \oplus ( f_1,f_2)$,
		as 	$D \oplus ( f_2,f_1)$ is the one with the vertex $x$ within $f_2$.
	}
\end{definition}

\begin{lemma}\label{use4-j}
Let $D$ be a $1$-planar drawing of 
a $1$-planar graph $G$,
and let $v\in V(G)$ is a simplicial vertex in $G$ of degree $4$.
Assume that $D':=D|(G-v)$ is triangulated
and $v$ is within face $f$ of $D'$. 
If $e$ is the only edge on $\partial (f)$ 
such that $f':=\OPF{f}{e}$ is 
uncrossed, then 
$D=D'\oplus (f,f')$.
\end{lemma} 

\begin{proof}
%	By Lemma~\ref{insert-v} (iii), 	$v$ lies in a uncrossed face 
Let $f=[bcab]$ and $f'=[bcdb]$, 
where edge $bc$ is on both 
$\partial(f)$ and $\partial(f')$. 
Since $v$ is within $f$ of $D'$, 
by Lemma~\ref{insert-v} (ii),
$N_D(v)=\{a,b,c,d\}$ is a clique of $D$. 
Thus, $D=D'\oplus (f,f')$.
\end{proof}

\section{Characterization of $4$-connected 
	$1$-planar chordal graphs}

Clearly, $K_5$ is the $4$-connected chordal 
graph with the smallest order,
and $K_6$ and $K_6-e$ 
are the only  
$4$-connected chordal 
graph of order  $6$.
Note that $\kappa(K_6)=5> 4$.
%In order to prove Theorem \ref{4_conn},we provide two results on $1$-planar chordal graphs. 
We first prove that $K_6$ is the only 
$1$-planar chordal graph 
with connectivity $5$, and 
every $1$-planar chordal graph $G$ with $\kappa(G)=4$ 
is a $4$-tree.

\begin{proposition}\label{main0}
  Let $G$ be a  $1$-planar chordal 
  graph that is not $K_6$. Then $\kappa(G)\le 4$. Moreover, if $\kappa(G)=4$, $G$ is a $4$-tree.
\end{proposition}

\begin{proof}
	Suppose $\kappa(G)\ge 5$. Clearly, $|V(G)|\ge 6$. 
	Since $G\ncong K_6$, we have 
	 $|V(G)|\ge 7$ . 
By Lemmas~\ref{K4n-8}
	and~\ref{k-connectedchordal}, we have 
	$$
	5|V(G)|-15\le |E(G)|\le 4|V(G)|-8, 
	$$
	 implying that $|V(G)|\le 7$. 
	 Thus, $|V(G)|= 7$. 
	 However, when $|V(G)|= 7$, 
	 applying Lemmas~\ref{K4n-8}
	 and~\ref{k-connectedchordal}
	 yields that 
	 $$
	 20=5\times 7-15 \le |E(G)| \le 
	 4\times 7 -9=19, 
	 $$
	 a contradiction. Hence 
	  $\kappa(G)\le 4$.
	
	Now assume that $\kappa(G)= 4$.
	 Then $|V(G)|\ge 5$. When $|V(G)|=5$,  $G$ is $K_5$
	 which is a $4$-tree.
	 If $|V(G)|=6$, 
	 by Lemma \ref{k-connectedchordal}, $|E(G)|\ge 4\times 6- \frac{4\times 5}{2}=14$. 
	 Since $|E(K_6)|=15$, we have $|E(G)|=14$. Hence, $G\cong K_6-e$.
	By Lemma \ref{k6_e}, $K_6-e$ is a 1-planar 4-tree. For $|V(G)|\ge 7$, since $\kappa(G)=4$, by Lemma \ref{K6conn}, $G$ has no 6-clique. Furthermore, by Lemma \ref{prop1}, $G$ is a $4$-tree.
\end{proof}

\iffalse 
\begin{proof}
Suppose  $\kappa(G)\ge 5$. Clearly, $|V(G)|\ge 6$. Moreover, $|V(G)|\ge 7$ since $G\ncong K_6$. By Lemma \ref{k-connectedchordal}, one has $|E(G)|\ge 5|V(G)|-15$.  By Lemma  \ref{K4n-8}, $|E(G)|\le 4|V(G)|-8$. Hence, $5|V(G)|-15\le 4|V(G)|-8$, which implies that $|V(G)|\le 7$. So $|V(G)|= 7$. If $|V(G)|=7$, by Lemma \ref{k-connectedchordal}, one has $|E(G)|\ge 5\times 7-\frac{1}{2}\times 5\times (5+1)= 20$. But by Lemma  \ref{K4n-8}, one has $|E(G)|\le 4\times 7 -9=19$, a contradiction. So $\kappa(G)\le 4$.

If $\kappa(G)= 4$, then $|V(G)|\ge 5$. If $|V(G)|=5$, then $G\cong K_5$. Clearly $K_5$ is a 4-tree and 1-planar. If $|V(G)|=6$,  then by Lemma \ref{k6_e}, $G\cong K_6\setminus e$ where $e\in E(K_6)$, and $K_6-e$ is a 1-planar 4-tree.  For $|V(G)|\ge 7$, since $\kappa(G)=4$, by Lemma \ref{K6conn},  $G$ has no 6-clique. Furthermore,  by Lemma \ref{prop1}, $G$ is a $4$-tree. 
\end{proof}
\fi 

%The \textit{join} $G+ H$ of two graphs $G$ and $H$ is obtained from vertex-disjoint copies of $G$ and $H$ by adding all edges joining every vertex in  $V (G)$ to every vertex in $V(H)$.

By Proposition  \ref{main0}, 
$K_6$ is the only 5-connected  1-planar chordal graphs, and all
1-planar chordal graphs with connectivity four are 4-trees.  
However, there exist 
4-trees which are not 1-planar. 
For example, 
 $K_4+\overline{K_3}$ is a 4-tree, 
 where $G+H$ 
 is the graph obtained from vertex disjoint copies of $G$ and 
 $H$
 by adding edges joining every vertex in  $G$ to every vertex in  $H$.
 But  $K_4+\overline{K_3}$ is not 1-planar (see \cite{Korzhik}).
  %where $\overline{H}$ denotes the complement of a graph $H$.}
  
 An \textit{Apollonian network} is a planar graph obtained by starting with a triangle, and repeatedly picking a triangular  face $f$ and subdividing $f$ into three triangles by inserting a new vertex in it.  
A planar graph is a 3-tree 
 if and only if it is a Apollonian network \cite{Biedl}. 
 In the following, we are going to
 establish an analogous result for 
 $4$-connected and 1-planar chordal graphs by the 4-join operation 
introduced in the previous 
section. 

By Proposition~\ref{main0}, every $1$-planar chordal graph $G$ with 
$\kappa(G)=4$ is a $4$-tree. 
Such graphs of orders at most $6$ 
are $K_5$ and $K_6-e$. 
In the following, we first characterize 
$1$-planar $4$-trees of order $7$, 
and then all $1$-planar $4$-trees of order more than $7$.

\begin{proposition}\label{4t-order7}
	The $1$-planar drawings of 
	all $1$-planar $4$-trees of order $7$ are shown in Fig.~\ref{4t-or7}.
\end{proposition}

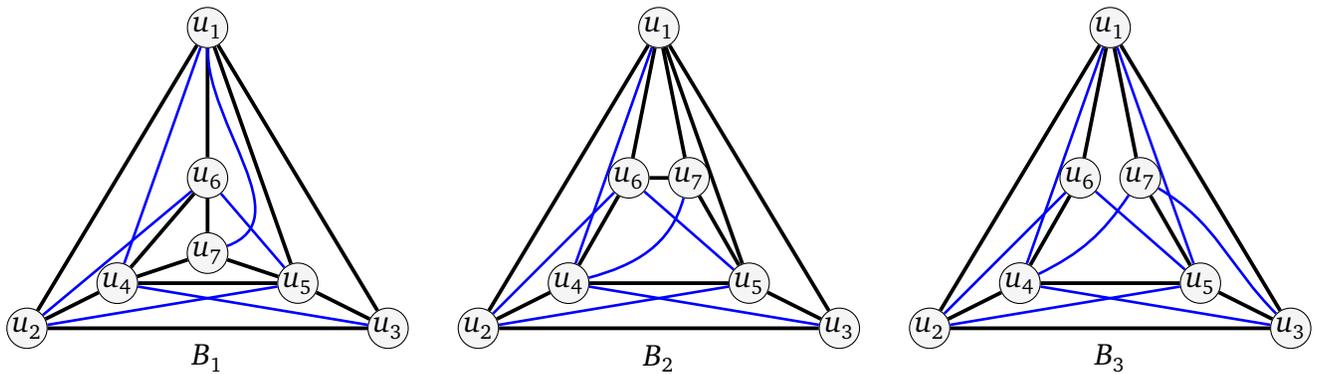
\begin{figure}[h!]
	\centering 

	\begin{tikzpicture}[scale=0.4]
		\tikzset{blueedge_thick/.style = {line width=1pt, blue, opacity=0.7}}
		\begin{pgfonlayer}{nodelayer}
			\node [style=none] (19) at (3, 0) {$B_2$};
			\node [style=none] (20) at (18, 0) {$B_3$};
			\node [style=whitenode] (24) at (-12, 11) {$u_1$};
			\node [style=whitenode] (25) at (-18, 1) {$u_2$};
			\node [style=whitenode] (26) at (-6, 1) {$u_3$};
			\node [style=whitenode] (27) at (-9, 2.5) {$u_5$};
			\node [style=whitenode] (28) at (-15, 2.5) {$u_4$};
			\node [style=whitenode] (29) at (-12, 6) {$u_6$};
			\node [style=none] (30) at (-12, 0) {$B_1$};
			\node [style=whitenode] (31) at (3, 11) {$u_1$};
			\node [style=whitenode] (32) at (-3, 1) {$u_2$};
			\node [style=whitenode] (33) at (9, 1) {$u_3$};
			\node [style=whitenode] (34) at (6, 2.5) {$u_5$};
			\node [style=whitenode] (35) at (0, 2.5) {$u_4$};
			\node [style=whitenode] (37) at (18, 11) {$u_1$};
			\node [style=whitenode] (38) at (12, 1) {$u_2$};
			\node [style=whitenode] (39) at (24, 1) {$u_3$};
		
\node [style=whitenode] (41) at (15, 2.5) {$u_4$};
\node [style=whitenode] (40) at (21, 2.5) {$u_5$};
			\node [style=whitenode] (36) at (2, 6) {$u_6$};
			\node [style=whitenode] (42) at (17, 6) {$u_6$};
			\node [style=whitenode] (43) at (-12, 3.5) {$u_7$};
			\node [style=whitenode] (45) at (4, 6) {$u_7$};
			\node [style=whitenode] (46) at (19, 6) {$u_7$};
		\end{pgfonlayer}
		\begin{pgfonlayer}{edgelayer}
			\draw [style={black_thick}] (24) to (25);
			\draw [style={black_thick}] (25) to (26);
			\draw [style={black_thick}] (26) to (24);
			\draw [style={black_thick}] (24) to (27);
			\draw [style={black_thick}] (27) to (26);
			\draw [style={black_thick}] (28) to (25);
			\draw [style={black_thick}] (28) to (27);
			\draw [style={black_thick}] (29) to (28);
			\draw [style={blueedge_thick}] (29) to (27);
			\draw [style={black_thick}] (24) to (29);
			\draw [style={blueedge_thick}] (24) to (28);
			\draw [style={blueedge_thick}] (29) to (25);
			\draw [style={blueedge_thick}] (25) to (27);
			\draw [style={blueedge_thick}] (28) to (26);
			\draw [style={black_thick}] (31) to (32);
			\draw [style={black_thick}] (32) to (33);
			\draw [style={black_thick}] (33) to (31);
			\draw [style={black_thick}] (31) to (34);
			\draw [style={black_thick}] (34) to (33);
			\draw [style={black_thick}] (35) to (32);
			\draw [style={black_thick}] (35) to (34);
			\draw [style={black_thick}] (36) to (35);
			\draw [style={blueedge_thick}] (36) to (34);
			\draw [style={black_thick}] (31) to (36);
			\draw [style={blueedge_thick}] (31) to (35);
			\draw [style={blueedge_thick}] (36) to (32);
			\draw [style={blueedge_thick}] (32) to (34);
			\draw [style={blueedge_thick}] (35) to (33);
			\draw [style={black_thick}] (37) to (38);
			\draw [style={black_thick}] (38) to (39);
			\draw [style={black_thick}] (39) to (37);
			\draw [style={blueedge_thick}] (37) to (40);
			\draw [style={black_thick}] (40) to (39);
			\draw [style={black_thick}] (41) to (38);
			\draw [style={black_thick}] (41) to (40);
			\draw [style={black_thick}] (42) to (41);
			\draw [style={blueedge_thick}] (42) to (40);
			\draw [style={black_thick}] (37) to (42);
			\draw [style={blueedge_thick}] (37) to (41);
			\draw [style={blueedge_thick}] (42) to (38);
			\draw [style={blueedge_thick}] (38) to (40);
			\draw [style={blueedge_thick}] (41) to (39);
			\draw [style={black_thick}] (29) to (43);
			\draw [style={black_thick}] (43) to (28);
			\draw [style={black_thick}] (43) to (27);
			\draw [style={black_thick}] (46) to (40);
			\draw [style={blueedge_thick}, bend left=15] (46) to (41);
			\draw [style={blueedge_thick}, in=135, out=-30] (46) to (39);
			\draw [style={black_thick}] (31) to (45);
			\draw [style={blueedge_thick}, in=15, out=-105] (45) to (35);
			\draw [style={black_thick}] (45) to (34);
			\draw [style={black_thick}] (36) to (45);
			\draw [style={black_thick}] (37) to (46);
			\draw [style={blueedge_thick}, in=20, out=-90] (24) to (43);
		\end{pgfonlayer}
	\end{tikzpicture}

	\caption{All $1$-planar drawings of 
		$1$-planar $4$-trees of order $7$}     
	\label{4t-or7}
\end{figure}

\begin{proof} Let $D$ be a $1$-planar drawing of a 
$1$-planar $4$-tree $T$ 
of order $7$, and let $v$ be 
	a simplicial vertex of $T$ with $d_T(v)=4$. 
	Then,  $N_T(v)$ is a clique in $T$ 
	of size $4$ and 
	$D|(T-v)$ must be a $1$-planar drawing
	of $T-v$. 
	Since $\kappa(T)=4$, 
	by Lemma~\ref{3cutface}, 
	$v$ cannot be within a crossed triangular face of $D|(T-v)$. 

\noindent {\bf Claim A}: 
	$D|(T-v)$ is isomorphic to $A_1$
	shown in 
	Fig.~\ref{K6_edrwings}
	or \ref{A2A3}

Note that $T-v\cong K_6-e$  
is a $1$-planar $4$-tree 
of order $6$. 
By Lemma~\ref{k6_e}, 
$D|(T-v)$ must be one of the 
three non-isomorphic 
$1$-planar drawings 
$A_1, A_2$ or $A_3$ 
shown in 
Fig.~\ref{K6_edrwings}
or \ref{A2A3}. 

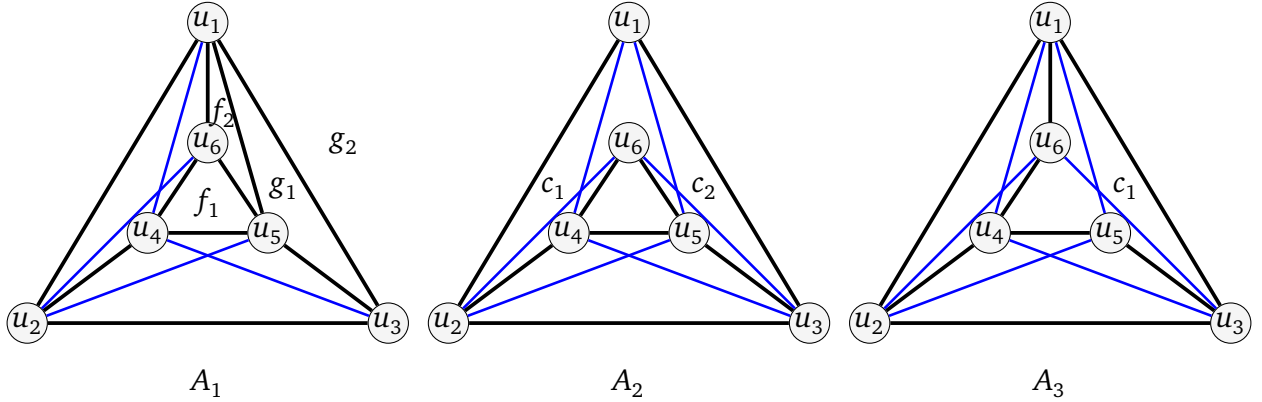
\begin{figure}[h!]
	\centering
	\begin{tikzpicture}[scale=0.8]
		\tikzset{blueedge_thick/.style = {line width=1pt, blue, opacity=0.7}}
		\node [style=whitenode] (5) at (0, 6) {$u_1$};
		\node [style=whitenode] (6) at (-3, 1) {$u_2$};
		\node [style=whitenode] (7) at (3, 1) {$u_3$};
		%\node [style=whitenode] (8) at (1, 2.5) {$u_6$};
		
		\node [style=whitenode] (8) at (1, 2.5) {$u_5$};
	\node [style=whitenode] (27) at (-6, 2.5) {$u_5$};
	\node [style=whitenode] (13) at (8, 2.5) {$u_5$};
	
	\node [style=whitenode] (14) at (6, 2.5) {$u_4$};
	\node [style=whitenode] (9) at (-1, 2.5) {$u_4$};
	\node [style=whitenode] (28) at (-8, 2.5) {$u_4$};
	\node [style=whitenode] (29) at (-7, 4) {$u_6$};
		\node [style=whitenode] (16) at (0, 4) {$u_6$};
		\node [style=whitenode] (17) at (7, 4) {$u_6$};
		\node [style=whitenode] (10) at (7, 6) {$u_1$};
		\node [style=whitenode] (11) at (4, 1) {$u_2$};
		\node [style=whitenode] (12) at (10, 1) {$u_3$};
		\node [style=none] (19) at (0, 0) {$A_2$};
		\node [style=none] (20) at (7, 0) {$A_3$};
		\node [style=none] (21) at (-1.25, 3.25) {$c_1$};
		\node [style=none] (22) at (1.25, 3.25) {$c_2$};
		\node [style=none] (23) at (8.25, 3.25) {$c_1$};
		\node [style=whitenode] (24) at (-7, 6) {$u_1$};
		\node [style=whitenode] (25) at (-10, 1) {$u_2$};
		\node [style=whitenode] (26) at (-4, 1) {$u_3$};
		\node [style=none] (30) at (-7, 0) {$A_1$};
		\node [style=none] (31) at (-7, 3) {$f_1$};
		\node [style=none] (32) at (-6.75, 4.5) {$f_2$};
		\node [style=none] (33) at (-5.75, 3.25) {$g_1$};
		\node [style=none] (34) at (-4.75, 4) {$g_2$};
		\draw [style={black_thick}] (5) to (6);
		\draw [style={black_thick}] (6) to (7);
		\draw [style={black_thick}] (7) to (5);
		\draw [style={black_thick}] (8) to (7);
		\draw [style={black_thick}] (9) to (6);
		\draw [style={black_thick}] (9) to (8);
		\draw [style={black_thick}] (10) to (11);
		\draw [style={black_thick}] (11) to (12);
		\draw [style={black_thick}] (12) to (10);
		\draw [style={black_thick}] (13) to (12);
		\draw [style={black_thick}] (14) to (11);
		\draw [style={black_thick}] (14) to (13);
		\draw [style={black_thick}] (16) to (8);
		\draw [style={black_thick}] (17) to (14);
		\draw [style={black_thick}] (17) to (10);
		\draw [style={black_thick}] (16) to (9);
		\draw [style={blueedge_thick}] (6) to (16);
		\draw [style={blueedge_thick}] (5) to (9);
		\draw [style={blueedge_thick}] (5) to (8);
		\draw [style={blueedge_thick}] (16) to (7);
		\draw [style={blueedge_thick}] (6) to (8);
		\draw [style={blueedge_thick}] (9) to (7);
		\draw [style={blueedge_thick}] (10) to (14);
		\draw [style={blueedge_thick}] (11) to (17);
		\draw [style={blueedge_thick}] (10) to (13);
		\draw [style={blueedge_thick}] (17) to (12);
		\draw [style={blueedge_thick}] (11) to (13);
		\draw [style={blueedge_thick}] (14) to (12);
		\draw [style={black_thick}] (24) to (25);
		\draw [style={black_thick}] (25) to (26);
		\draw [style={black_thick}] (26) to (24);
		\draw [style={black_thick}] (24) to (27);
		\draw [style={black_thick}] (27) to (26);
		\draw [style={black_thick}] (28) to (25);
		\draw [style={black_thick}] (28) to (27);
		\draw [style={black_thick}] (29) to (28);
		\draw [style={black_thick}] (29) to (27);
		\draw [style={black_thick}] (24) to (29);
		\draw [style={blueedge_thick}] (24) to (28);
		\draw [style={blueedge_thick}] (29) to (25);
		\draw [style={blueedge_thick}] (25) to (27);
		\draw [style={blueedge_thick}] (28) to (26);
	\end{tikzpicture}
	
	\caption{$A_1$, $A_2$ and $A_3$}     
	\label{A2A3}
\end{figure}

We shall show that $D|(T-v)$ must be the $1$-planar drawing $A_1$.
If $D|(T-v)$ is $A_2$, as shown in 
Fig.~\ref{A2A3}, 
then $v$ must be within a uncrossed 
triangular face of $A_2$ or 
a face of $A_2$ which is not triangular,
implying that  
$v$ must be within 
the crossed face 
$[u_1c_1u_6c_2u_1]$,
or within one of the uncrossed faces 
$[u_1u_2u_3u_1],  [u_4u_5u_6u_4]$,
where $c_1$ and $c_2$ are crossings of edges, as shown in 
$A_2$ of Fig.~\ref{A2A3}.
However,  by Lemma \ref{insert-v} (i), 
in any one of three situations,
connecting $v$ to any one clique in 
$D|(T-v)$ of size $4$ will 
violate the 1-planarity of $T$, 
a contradiction. 

Similarly,  $D|(T-v)$ cannot be $A_3$.
Hence  Claim A holds.
	
	Now we are going to show that $D$ is isomorphic to one of the drawings shown in Fig.~\ref{4t-or7}. 
	By Claim A, $D$ is obtained 
	by inserting $v$ to a face of 
	$A_1$ and then joining $v$ 
	to all vertices in a $4$-clique of
	$A_1$.

	Note that $A_1$ is triangular.
	By Lemma~\ref{insert-v} (i), 
	$v$ cannot  be  within any crossed triangular face of $A_1$. 
	Thus, $v$ must be in a triangular 
	face of $A_1$, i.e., $f_1, f_2, g_1$ 
	or $g_2$ (see Fig. ~\ref{A2A3}). 
	
	\noindent {\bf Case 1}: $v$ is within 
	face $f_1=[u_4,u_5,u_6,u_4]$. 
	
	In this case, 
	by Lemma~\ref{insert-v} (ii), 
	$N_T(v)$
	must be the set $\{u_4,u_5,u_6,u_1\}$,
	and $D$ is isomorphic to $B_1$
	shown in Fig.~\ref{4t-or7}.  
	
	\noindent {\bf Case 2}: $v$ is within 
	face $f_2=[u_1,u_5,u_6,u_1]$. 
	
	In this case, 
	by Lemma~\ref{insert-v} (ii), 
		$N_T(v)$
	must be either $\{u_1, u_4,u_5,u_6,u_1\}$
	or $\{u_1, u_3,u_4, u_5,u_1\}$, 
	and $D$ is isomorphic to either 
	$B_2$ or $B_3$
	shown in Fig.~\ref{4t-or7}.  
	
	\noindent {\bf Case 3}: $v$ is within face $g_1=[u_1,u_3,u_5]$.
	
	In this case, 	
	by Lemma~\ref{insert-v} (ii), 
	$N_T(v)$
	must be either $\{u_1, u_3,u_5,u_2\}$
	or $\{u_1, u_3,u_5,u_4\}$, 
	and thus $D$ is isomorphic to either 
	$B_2$ or $B_3$
	shown in Fig.~\ref{4t-or7}.  
	
	\noindent {\bf Case 4}: $v$ is within 
	face $g_2=[u_1,u_2,u_3]$. 
	
	In this case,  	
	by Lemma~\ref{insert-v} (ii), 
	$N_T(v)$
	must be the set $\{u_1,u_2,u_3,u_5\}$,
	and $D$ is thus isomorphic to $B_1$
	shown in Fig.~\ref{4t-or7}.  
	
	Hence the result holds.
\end{proof}

	Let $\Phi$ denote the set of 
	$1$-planar drawings of 
	some $1$-planar graphs defined below: 
	\begin{enumerate}
		\item 
		the $1$-planar drawings
		$B_1$ and $B_2$ in Fig.~\ref{4t-or7}
		are the ones in $\Phi$ 
		with the smallest order; and 
		\item 
		if $D\in \Phi$, then both 
		$D\oplus (f_1,f_2)$
		and $D\oplus (f_2,f_1)$ 
		belong to 
		 $\Phi$ 
		for twin faces $f_1$ and $f_2$ in $D$.
	\end{enumerate}

\begin{lemma}\label{set-phi}
Every $1$-planar drawing $D$ 
in $\Phi$ represents
a $1$-planar $4$-tree 
with exactly two simplicial vertices. 
\end{lemma}

\begin{proof}
We have  the following two facts:
\begin{enumerate}
	\item Both $B_1$ and $B_2$,
	shown in Fig.~\ref{4t-or7},  
	represent $1$-planar $4$-trees
	with exactly two somplicial vertices
	and two pairs of twin faces, 
	and each 
	simplicial vertex is 
	located in one pair of twin faces;
	\item if $D=D'\oplus (f_1,f_2)$ 
	for a pair of twin faces $f_1$ and $f_2$ of $D'$
	and some simplicial vertex of $D$ 
	is on $\partial(f_1)\cup \partial(f_2)$, 
	then $D$ and $D'$ have the equal number 
	of simplicial vertices. 
\end{enumerate}
By the definition of $\Phi$, the result holds.
\end{proof}

In the following, we will show that 
the $1$-planar drawings of 
all $1$-planar $4$-trees
of order at least $8$ 
belongs to $\Phi$.

For any $1$-planar drawing $D$
with at least one uncrossed face, 
let $E_{uf}(D)$ be the set of edges $e$ in $D$ such that $e$  is on $\partial(f)$ 
for some uncrossed face $f$ of $D$,
and let $V_{uf}$ be the set of vertices in $D$ which 
are incident with some edges in 
$E_{uf}(D)$.
Let $D_{uf}$ denote 
the $1$-planar drawing
$D|H$, where 
$H=(V_{uf},E_{uf}(D))$.

 \iffalse 
obtained from $D$ by removing all 
edges in $E(D)\setminus E_{uf}(D)$,
and then  removing all 
isolated vertices in the $1$-planar drawing obtained by removing 
all edges in $E(D)\setminus E_{uf}(D)$.
If $D$ is a plane graph, then 
$D_{uf}$ is $D$ itself. 
\fi

\begin{proposition}\label{main-4t}
Any $1$-planar drawing of 
a $1$-planar $4$-tree of order at least $8$ is isomorphic to 
some $1$-planar 
drawing in $\Phi$.
\end{proposition}

\begin{proof}
Let $T$ be a $1$-planar $4$-tree of order $n$ ($\ge 8$),
and let $v$ be a simplicial vertex of $T$. 
Thus,  $N_T(v)$ is a $4$-clique in $T$. 
Let $D$ be a $1$-planar drawing of $T$.
Then, 
$D|(T-v)$ must be a $1$-planar drawing
of $T-v$. 
Since $\kappa(T)=4$, 
by Lemma~\ref{3cutface}, 
$v$ cannot be within a crossed triangular face of $D|(T-v)$. 

In the following, we shall show that 
$D_{uf}$ is isomorphic to some $D_j$
in the set $\{D_j: 1\le j\le 6\}$,  shown in 
Fig.~\ref{Duf} with the following 
properties:
\begin{enumerate}
	\item [P1.] in each $D_j$
	($1\le j\le 6$), 
$f_i$ and $f_{i+1}$ are 
twin faces\footnote{The set of vertices in 
	$\partial(f_i)\cup \partial(f_{i+1})$,
	where $i\in \{1,3\}$, 
	forms a $4$-clique of $D$,
	although it does not in $D_{uf}$.
} of $D$ for each $i\in \{1,3\}$,
where $f_1, f_2, f_3$ and $f_4$ are faces in $D_j$, 
and 
\item [P2.]
$v_0v_2$, $v_1v_2$ and $v_1v_3$
are not edges in $D$ 
if $D_{uf}$  is $D_1$ or $D_2$.\footnote{
In each $D_i$, $1\le i\le 2$,
Property 2
 implies that  
 there are two nonadjacent 
 vertices in $D_i$ of degree $3$
 each of which is not adjacent to some 
vertex in $D_i$ 
of degree $2$ in $D$.
P2 also implies that if $v$ is within face $f_2$, then $v_2\notin N_D(v)$.
}
\end{enumerate}
\begin{figure}[h!]
	\centering 
\begin{tikzpicture}[scale=0.5]
\tikzset{whitenode/.style={fill={rgb,255: red,245; green,245; blue,245}, draw=black, shape=circle, minimum size=0.1cm, inner sep=0.1pt,scale=0.7}}
\tikzset{whitenode1/.style={fill={rgb,255: red,245; green,245; blue,245}, draw=black, shape=circle, minimum size=0.1cm, inner sep=0.1pt,scale=2.5}}
\tikzset{shadow_/.style={fill={rgb,255: red,206; green,226; blue,248}, draw={none}, line width=0.4mm}}
\tikzset{shadow_pure/.style={fill={rgb,255: red,224; green,182; blue,250}, draw={none}, line width=0.4mm}}
	\begin{pgfonlayer}{nodelayer}
		\node [style=whitenode] (0) at (-4.25, 2.5) {$v_0$};
		\node [style=whitenode] (1) at (-2, 5) {$v_1$};
		\node [style=whitenode1] (2) at (-2, 0) {};
		\node [style=whitenode1] (3) at (1, 5) {};
		\node [style=whitenode] (4) at (3, 2.5) {$v_2$};
		\node [style=whitenode] (5) at (1, 0) {$v_3$};
		\node [style=none] (6) at (-3.25, 2.5) {$f_1$};
		\node [style=none] (7) at (-1, 3.75) {$f_2$};
		\node [style=none] (8) at (0, 1.75) {$f_3$};
		\node [style=none] (9) at (1.75, 2.3) {$f_4$};
		\node [style=whitenode] (11) at (-15.25, 2.5) {$v_0$};
		\node [style=whitenode] (12) at (-13, 5) {$v_1$};
		\node [style=whitenode1] (13) at (-13, 0) {};
		\node [style=whitenode1] (14) at (-10, 5) {};
		\node [style=whitenode] (15) at (-8, 2.5) {$v_2$};
		\node [style=whitenode] (16) at (-10, 0) {$v_3$};
		\node [style=none] (17) at (-14.25, 2.5) {$f_1$};
		\node [style=none] (18) at (-12, 3.75) {$f_2$};
		\node [style=none] (19) at (-10.5, 2.5) {$f_3$};
		\node [style=none] (20) at (-10.5, 0.75) {$f_4$};
		\node [style=whitenode1] (21) at (9, 5) {};
		\node [style=whitenode1] (22) at (6.75, 2.5) {};
		\node [style=whitenode1] (23) at (9, 0) {};
		\node [style=whitenode1] (24) at (11.25, 2.5) {};
		\node [style=whitenode1] (25) at (14.5, 5) {};
		\node [style=whitenode1] (26) at (12.25, 2.5) {};
		\node [style=whitenode1] (27) at (14.5, 0) {};
		\node [style=whitenode1] (28) at (16.75, 2.5) {};
		\node [style=none] (29) at (8, 2.5) {$f_1$};
		\node [style=none] (30) at (9.75, 2.5) {$f_2$};
		\node [style=none] (31) at (13.75, 2.5) {$f_3$};
		\node [style=none] (32) at (15.3, 2.5) {$f_4$};
		\node [style=whitenode1] (33) at (-13.5, -4) {};
		\node [style=whitenode1] (34) at (-15.75, -6.5) {};
		\node [style=whitenode1] (35) at (-13.5, -9) {};
		\node [style=whitenode1] (36) at (-11.25, -6.5) {};
		\node [style=whitenode1] (37) at (-9, -4) {};
		\node [style=whitenode1] (38) at (-11.25, -6.5) {};
		\node [style=whitenode1] (39) at (-9, -9) {};
		\node [style=whitenode1] (40) at (-6.75, -6.5) {};
		\node [style=none] (41) at (-14.5, -6.5) {$f_1$};
		\node [style=none] (42) at (-12.75, -6.5) {$f_2$};
		\node [style=none] (43) at (-9.75, -6.5) {$f_3$};
		\node [style=none] (44) at (-8.25, -6.5) {$f_4$};
		\node [style=whitenode1] (45) at (-2.5, -4) {};
		\node [style=whitenode1] (46) at (-4.75, -6.5) {};
		\node [style=whitenode1] (47) at (-2.5, -9) {};
		\node [style=whitenode1] (48) at (-0.25, -6.5) {};
		\node [style=whitenode1] (49) at (2, -4) {};
		\node [style=whitenode1] (50) at (-0.25, -6.5) {};
		\node [style=whitenode1] (51) at (2, -9) {};
		\node [style=whitenode1] (52) at (4.25, -6.5) {};
		\node [style=none] (53) at (-3.5, -6.5) {$f_1$};
		\node [style=none] (54) at (-1.75, -6.5) {$f_2$};
		\node [style=none] (55) at (2, -5.5) {$f_3$};
		\node [style=none] (56) at (2, -7.5) {$f_4$};
		\node [style=whitenode1] (57) at (11, -6.75) {};
		\node [style=whitenode1] (58) at (8.5, -4.5) {};
		\node [style=whitenode1] (59) at (6, -6.75) {};
		\node [style=whitenode1] (60) at (8.5, -9) {};
		\node [style=whitenode1] (61) at (16, -6.75) {};
		\node [style=whitenode1] (62) at (13.5, -4.5) {};
		\node [style=whitenode1] (63) at (11, -6.75) {};
		\node [style=whitenode1] (64) at (13.5, -9) {};
		\node [style=none] (65) at (8.5, -5.75) {$f_1$};
		\node [style=none] (66) at (8.5, -7.5) {$f_2$};
		\node [style=none] (67) at (13.5, -5.75) {$f_3$};
		\node [style=none] (68) at (13.5, -7.5) {$f_4$};
		\node [style=none] (69) at (-4.75, -1.5) {};
		\node [style=none] (70) at (-12, -2) {$D_1$};
		\node [style=none] (71) at (-1, -2) {$D_2$};
		\node [style=none] (72) at (11.75, -2) {$D_3$};
		\node [style=none] (73) at (-11.75, -10) {$D_4$};
		\node [style=none] (74) at (0, -10) {$D_5$};
		\node [style=none] (75) at (11.5, -10) {$D_6$};
	\end{pgfonlayer}
	\begin{pgfonlayer}{edgelayer}
		\draw [style={shadow_}] (0.center)
			 to (2.center)
			 to (3.center)
			 to [bend left=360, looseness=1.25] (1.center)
			 to cycle;
		\draw [style={black_thick}] (1) to (2);
		\draw [style={black_thick}] (1) to (0);
		\draw [style={black_thick}] (0) to (2);
		\draw [style={black_thick}] (1) to (3);
		\draw [style={shadow_pure}] (4.center)
			 to (5.center)
			 to (2.center)
			 to (3.center)
			 to cycle;
		\draw [style={black_thick}] (3) to (2);
		\draw [style={black_thick}] (2) to (5);
		\draw [style={black_thick}] (5) to (3);
		\draw [style={black_thick}] (3) to (4);
		\draw [style={black_thick}] (4) to (5);
		\draw [style={shadow_}] (11.center)
			 to (13.center)
			 to (14.center)
			 to [bend left=360, looseness=1.25] (12.center)
			 to cycle;
		\draw [style={black_thick}] (12) to (13);
		\draw [style={black_thick}] (12) to (11);
		\draw [style={black_thick}] (11) to (13);
		\draw [style={black_thick}] (12) to (14);
		\draw [style={shadow_pure}] (15.center)
			 to (16.center)
			 to (13.center)
			 to (14.center)
			 to cycle;
		\draw [style={black_thick}] (14) to (13);
		\draw [style={black_thick}] (13) to (16);
		\draw [style={black_thick}] (14) to (15);
		\draw [style={black_thick}] (15) to (16);
		\draw [style={black_thick}] (13) to (15);
		\draw [style={shadow_}] (21.center)
			 to (22.center)
			 to (23.center)
			 to (24.center)
			 to cycle;
		\draw [style={shadow_pure}] (25.center)
			 to (26.center)
			 to (27.center)
			 to (28.center)
			 to cycle;
		\draw [style={black_thick}] (21) to (22);
		\draw [style={black_thick}] (23) to (22);
		\draw [style={black_thick}] (23) to (24);
		\draw [style={black_thick}] (24) to (21);
		\draw [style={black_thick}] (21) to (23);
		\draw [style={black_thick}] (26) to (25);
		\draw [style={black_thick}] (25) to (28);
		\draw [style={black_thick}] (28) to (27);
		\draw [style={black_thick}] (27) to (26);
		\draw [style={black_thick}] (25) to (27);
		\draw [style={shadow_}] (33.center)
			 to (34.center)
			 to (35.center)
			 to (36.center)
			 to cycle;
		\draw [style={shadow_pure}] (37.center)
			 to (38.center)
			 to (39.center)
			 to (40.center)
			 to cycle;
		\draw [style={black_thick}] (33) to (34);
		\draw [style={black_thick}] (35) to (34);
		\draw [style={black_thick}] (35) to (36);
		\draw [style={black_thick}] (36) to (33);
		\draw [style={black_thick}] (33) to (35);
		\draw [style={black_thick}] (38) to (37);
		\draw [style={black_thick}] (37) to (40);
		\draw [style={black_thick}] (40) to (39);
		\draw [style={black_thick}] (39) to (38);
		\draw [style={black_thick}] (37) to (39);
		\draw [style={shadow_}] (45.center)
			 to (46.center)
			 to (47.center)
			 to (48.center)
			 to cycle;
		\draw [style={shadow_pure}] (49) to (50);
		\draw [style={shadow_pure}] (50) to (51);
		\draw [style={shadow_pure}, in=225, out=45] (51) to (52);
		\draw [style={shadow_pure}] (52) to (49);
		\draw [style={black_thick}] (45) to (46);
		\draw [style={black_thick}] (47) to (46);
		\draw [style={black_thick}] (47) to (48);
		\draw [style={black_thick}] (48) to (45);
		\draw [style={black_thick}] (45) to (47);
		\draw [style={shadow_pure}] (52.center)
			 to (49.center)
			 to (50.center)
			 to (51.center)
			 to cycle;
		\draw [style={black_thick}] (49) to (50);
		\draw [style={black_thick}] (50) to (51);
		\draw [style={black_thick}] (51) to (52);
		\draw [style={black_thick}] (52) to (49);
		\draw [style={black_thick}] (50) to (52);
		\draw [style={shadow_}] (57.center)
			 to (58.center)
			 to (59.center)
			 to (60.center)
			 to cycle;
		\draw [style={shadow_pure}] (61.center)
			 to (62.center)
			 to (63.center)
			 to (64.center)
			 to cycle;
		\draw [style={black_thick}] (57) to (58);
		\draw [style={black_thick}] (59) to (58);
		\draw [style={black_thick}] (59) to (60);
		\draw [style={black_thick}] (60) to (57);
		\draw [style={black_thick}] (57) to (59);
		\draw [style={black_thick}] (62) to (61);
		\draw [style={black_thick}] (61) to (64);
		\draw [style={black_thick}] (64) to (63);
		\draw [style={black_thick}] (63) to (62);
		\draw [style={black_thick}] (61) to (63);
	\end{pgfonlayer}
\end{tikzpicture}

%	\includegraphics[width=17 cm]
%	{Duf-1.eps}
	
%	(a) $D_1$ 	\hspace{4 cm} 	(b) $D_2$ 		\hspace{4 cm} 	(c) $D_3$ 

	\caption{Possible 
		$1$-planar drawings of $D'_{uf}$
		with properties P1 and P2
	}     
	\label{Duf}
\end{figure}
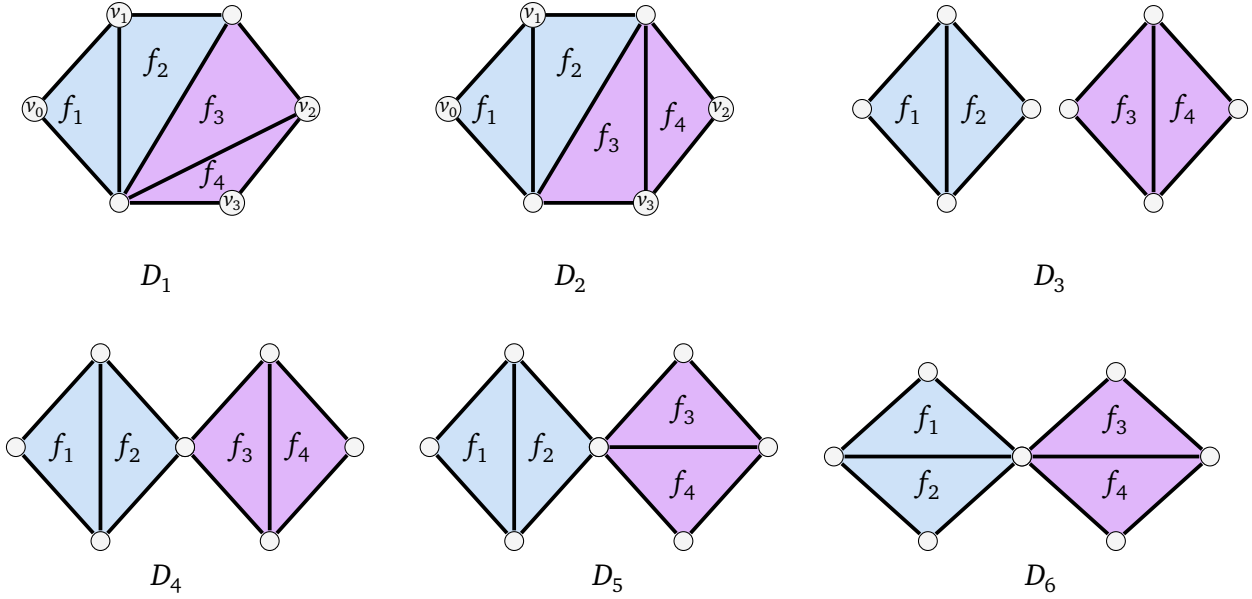

We first prove the following 
conclusion. 
Let $D':=D|(T-v)$. 

\begin{claim}\label{nc-1}
	If $n=8$, then $D'$ is triangulated 
	and 
	$D'_{uf}$ is isomorphic to $D_1$ or $D_4$ in Fig.~\ref{Duf}
	with properties P1 and P2.
\end{claim}

Note that $D'$ is a $1$-planar drawing 
of $T-v$. 
By Proposition~\ref{4t-order7}, 
$D'$ is isomorphic to $B_1, B_2$ or $B_3$ in Fig.~\ref{4t-or7}. 
By Lemma~\ref{insert-v} (i), 
$v$ cannot be within any face of $B_3$.
Thus, 
$D'$ is isomorphic to $B_1$ or $B_2$.
Clearly, $D'$ is triangulated. 

If $D'$ is isomorphic to $B_1$,
then $D'_{uf}$ is isomorphic to $D_4$ 
in  Fig.~\ref{Duf}.
If $D'$ is isomorphic to $B_2$,
then $D'_{uf}$ is isomorphic to $D_1$ 
in  Fig.~\ref{Duf} with the property that 
there is no edge in $D'$ 
joining $v_2$ to $v_0$ or $v_1$,
or $v_3$ to $v_1$. 
In both case, properties P1 and P2 hold.

Claim~\ref{nc-1} holds.

\begin{claim}\label{nc0}
	If $D'$ is triangulated and 
$D'_{uf}$ is isomorphic to $D_i$ 
	in Fig.~\ref{Duf} for some $i\in \{1,2,\cdots,6\}$ with properties 
	P1 and P2, 
	then $D$ is obtained from $D'$ 
	by a 4-join operation and 
	$D_{uf}$ is also 
	isomorphic to 
	isomorphic to some $D_i$ 
	in Fig.~\ref{Duf}, where $i\in \{1,2,\dots,6\}$,
	with properties 
	P1 and P2.
\end{claim}

By Lemma~\ref{insert-v} (ii), 
$v$ is within a uncrossed face of $D'$. 

\noindent {\bf Case 1}: 
$D'_{uf}$ is isomorphic to $D_1$.

If $v$ is within $f_1$ or $f_4$, 
then, by Lemma~\ref{use4-j},
$D=D'\oplus (f_1,f_2)$ 
or $D=D'\oplus (f_4,f_3)$ respectively. 
Clearly, $D_{uf}$ is isomorphic to $D_5$ and properties P1 and P2 hold
(see Fig.~\ref{uc-faces-3} (a)).

If $v$ is within $f_2$ or $f_3$, 
then, by Lemma~\ref{use4-j},
$D=D'\oplus (f_2,f_1)$ 
or $D=D'\oplus (f_3,f_4)$ respectively. 
Clearly, $D_{uf}$ is isomorphic to $D_2$ and properties P1 and P2 hold
(see Fig.~\ref{uc-faces-3} (b)).

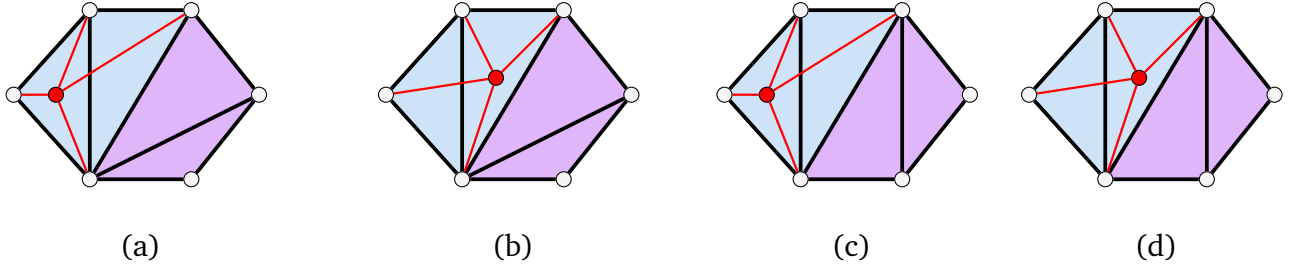
\begin{figure}[h!]
	\centering 
	\begin{tikzpicture}[scale=0.45]
		\tikzset{whitenode/.style={fill={rgb,255: red,245; green,245; blue,245}, draw=black, shape=circle, minimum size=0.2cm, inner sep=0.1pt}}
		\tikzset{shadow_/.style={fill={rgb,255: red,206; green,226; blue,248}, draw={none}, line width=0.4mm}}
		\tikzset{shadow_pure/.style={fill={rgb,255: red,224; green,182; blue,250}, draw={none}, line width=0.4mm}}
		\begin{pgfonlayer}{nodelayer}
			\node [style=whitenode] (0) at (-9.25, 2.5) {};
			\node [style=whitenode] (1) at (-7, 5) {};
			\node [style=whitenode] (2) at (-7, 0) {};
			\node [style=whitenode] (3) at (-4, 5) {};
			\node [style=whitenode] (4) at (-2, 2.5) {};
			\node [style=whitenode] (5) at (-4, 0) {};
			\node [style=whitenode] (10) at (-20.25, 2.5) {};
			\node [style=whitenode] (11) at (-18, 5) {};
			\node [style=whitenode] (12) at (-18, 0) {};
			\node [style=whitenode] (13) at (-15, 5) {};
			\node [style=whitenode] (14) at (-13, 2.5) {};
			\node [style=whitenode] (15) at (-15, 0) {};
			\node [style=none] (20) at (-16.5, -2) {(a)};
			\node [style=none] (21) at (-5.5, -2) {(b)};
			\node [style=whitenode] (23) at (0.75, 2.5) {};
			\node [style=whitenode] (24) at (3, 5) {};
			\node [style=whitenode] (25) at (3, 0) {};
			\node [style=whitenode] (26) at (6, 5) {};
			\node [style=whitenode] (27) at (8, 2.5) {};
			\node [style=whitenode] (28) at (6, 0) {};
			\node [style=none] (35) at (4.5, -2) {(c)};
			\node [style=whitenode] (36) at (9.75, 2.5) {};
			\node [style=whitenode] (37) at (12, 5) {};
			\node [style=whitenode] (38) at (12, 0) {};
			\node [style=whitenode] (39) at (15, 5) {};
			\node [style=whitenode] (40) at (17, 2.5) {};
			\node [style=whitenode] (41) at (15, 0) {};
			\node [style=none] (46) at (13.5, -2) {(d)};
			\node [style=rednode] (47) at (-19, 2.5) {};
			\node [style=rednode] (48) at (-6, 3) {};
			\node [style=rednode] (49) at (2, 2.5) {};
			\node [style=rednode] (50) at (13, 3) {};
		\end{pgfonlayer}
		\begin{pgfonlayer}{edgelayer}
			\draw [style={shadow_}] (0.center)
			to (2.center)
			to (3.center)
			to [bend left=360, looseness=1.25] (1.center)
			to cycle;
			\draw [style={black_thick}] (1) to (2);
			\draw [style={black_thick}] (1) to (0);
			\draw [style={black_thick}] (0) to (2);
			\draw [style={black_thick}] (1) to (3);
			\draw [style={shadow_pure}] (4.center)
			to (5.center)
			to (2.center)
			to (3.center)
			to cycle;
			\draw [style={black_thick}] (3) to (2);
			\draw [style={black_thick}] (2) to (5);
			%\draw [style={black_thick}] (5) to (3);
			\draw [style={black_thick}] (3) to (4);
			\draw [style={black_thick}] (4) to (5);
			\draw [style={black_thick}] (2) to (4);
			\draw [style={shadow_}] (10.center)
			to (12.center)
			to (13.center)
			to [bend left=360, looseness=1.25] (11.center)
			to cycle;
			\draw [style={black_thick}] (11) to (12);
			\draw [style={black_thick}] (11) to (10);
			\draw [style={black_thick}] (10) to (12);
			\draw [style={black_thick}] (11) to (13);
			\draw [style={shadow_pure}] (14.center)
			to (15.center)
			to (12.center)
			to (13.center)
			to cycle;
			\draw [style={black_thick}] (13) to (12);
			\draw [style={black_thick}] (12) to (15);
			\draw [style={black_thick}] (13) to (14);
			\draw [style={black_thick}] (14) to (15);
			\draw [style={black_thick}] (12) to (14);
			\draw [style={shadow_}] (23.center)
			to (25.center)
			to (26.center)
			to [bend left=360, looseness=1.25] (24.center)
			to cycle;
			\draw [style={black_thick}] (24) to (25);
			\draw [style={black_thick}] (24) to (23);
			\draw [style={black_thick}] (23) to (25);
			\draw [style={black_thick}] (24) to (26);
			\draw [style={shadow_pure}] (27.center)
			to (28.center)
			to (25.center)
			to (26.center)
			to cycle;
			\draw [style={black_thick}] (26) to (25);
			\draw [style={black_thick}] (25) to (28);
			\draw [style={black_thick}] (28) to (26);
			\draw [style={black_thick}] (26) to (27);
			\draw [style={black_thick}] (27) to (28);
			\draw [style={shadow_}] (36.center)
			to (38.center)
			to (39.center)
			to [bend left=360, looseness=1.25] (37.center)
			to cycle;
			\draw [style={black_thick}] (37) to (38);
			\draw [style={black_thick}] (37) to (36);
			\draw [style={black_thick}] (36) to (38);
			\draw [style={black_thick}] (37) to (39);
			\draw [style={shadow_pure}] (40.center)
			to (41.center)
			to (38.center)
			to (39.center)
			to cycle;
			\draw [style={black_thick}] (39) to (38);
			\draw [style={black_thick}] (38) to (41);
			\draw [style={black_thick}] (41) to (39);
			\draw [style={black_thick}] (39) to (40);
			\draw [style={black_thick}] (40) to (41);
			\draw [style=rededge] (10) to (47);
			\draw [style=rededge] (47) to (11);
			\draw [style=rededge] (47) to (12);
			\draw [style=rededge] (47) to (13);
			\draw [style=rededge] (0) to (48);
			\draw [style=rededge] (1) to (48);
			\draw [style=rededge] (48) to (2);
			\draw [style=rededge] (48) to (3);
			\draw [style=rededge] (24) to (49);
			\draw [style=rededge] (49) to (23);
			\draw [style=rededge] (49) to (25);
			\draw [style=rededge] (49) to (26);
			\draw [style=rededge] (37) to (50);
			\draw [style=rededge] (50) to (38);
			\draw [style=rededge] (50) to (39);
			\draw [style=rededge] (50) to (36);
		\end{pgfonlayer}
	\end{tikzpicture}
	
	%\includegraphics[width=17 cm]
	%{uc-faces-3.eps}
	
	%(a) \hspace{4 cm} (b)
	%\hspace{4 cm} (c)
	%\hspace{4 cm} (d)
	
	\caption{$D$ is obtained from $D'$ 
		by a $4$-join operation}     
	\label{uc-faces-3}
\end{figure}

\noindent {\bf Case 2}: 
$D'_{uf}$ is isomorphic to $D_2$.

If $v$ is within $f_1$ or $f_4$, 
then, by Lemma~\ref{use4-j},
$D=D'\oplus (f_1,f_2)$ 
or $D=D'\oplus (f_4,f_3)$ respectively. 
Clearly, $D_{uf}$ is isomorphic to $D_4$ and properties P1 and P2 hold
(see Fig.~\ref{uc-faces-3} (c)).

If $v$ is within $f_2$ or $f_3$, 
then, by Lemma~\ref{use4-j},
$D=D'\oplus (f_2,f_1)$ 
or $D=D'\oplus (f_3,f_4)$ respectively. 
Clearly, $D_{uf}$ is isomorphic to $D_1$ and properties P1 and P2 hold
(see Fig.~\ref{uc-faces-3} (d)).

\noindent {\bf Case 3}: 
$D'_{uf}$ is isomorphic to $D_i$ 
in Fig.~\ref{Duf} for some $i\in \{3,4,5,6\}$.

By Lemma~\ref{use4-j},
$D=D'\oplus (f_i,f_{i+1})$ 
or $D=D'\oplus (f_{i+1},f_{i})$ 
for some $i\in \{1,3\}$. Observe that 
\begin{itemize}[nolistsep]
	\item if $D'_{uf}$ is isomorphic to $D_3$, 
then $D_{uf}$ is also 
isomorphic to $D_3$ (see Fig.~\ref{uc-faces} (a));
\item 
if $D'_{uf}$ is isomorphic to $D_4$, 
then $D_{uf}$ is 
isomorphic to $D_3$ or $D_5$
(see Fig.~\ref{uc-faces} (b) and (c));
\item 
if $D'_{uf}$ is isomorphic to $D_5$, 
then $D_{uf}$ is 
isomorphic to $D_3$, $D_4$ or $D_6$
(see Fig.~\ref{uc-faces} (d), (e) and (f)),
and 
\item 
if $D'_{uf}$ is isomorphic to $D_6$, 
then $D_{uf}$ is 
isomorphic to $D_5$
(see Fig.~\ref{uc-faces} (g)).
\end{itemize}
In all cases above, properties P1 and P2 hold.

\begin{figure}[h!]
	\centering 
\tikzset{whitenode/.style={fill={rgb,255: red,245; green,245; blue,245}, draw=black, shape=circle, minimum size=0.2cm, inner sep=0.1pt}}
\tikzset{shadow_/.style={fill={rgb,255: red,206; green,226; blue,248}, draw={none}, line width=0.4mm}}
\tikzset{shadow_pure/.style={fill={rgb,255: red,224; green,182; blue,250}, draw={none}, line width=0.4mm}}
\begin{tikzpicture}[scale=0.38]
	\begin{pgfonlayer}{nodelayer}
		\node [style=whitenode] (0) at (-15.25, 5.5) {};
		\node [style=whitenode] (1) at (-17.5, 3) {};
		\node [style=whitenode] (2) at (-15.25, 0.5) {};
		\node [style=whitenode] (3) at (-13, 3) {};
		\node [style=whitenode] (4) at (-9.75, 5.5) {};
		\node [style=whitenode] (5) at (-12, 3) {};
		\node [style=whitenode] (6) at (-9.75, 0.5) {};
		\node [style=whitenode] (7) at (-7.5, 3) {};
		\node [style=none] (8) at (-12.75, -1.5) {(a)};
		\node [style=rednode] (9) at (-16.25, 3) {};
		\node [style=whitenode] (10) at (-2.25, 5.5) {};
		\node [style=whitenode] (11) at (-4.5, 3) {};
		\node [style=whitenode] (12) at (-2.25, 0.5) {};
		\node [style=whitenode] (13) at (0, 3) {};
		\node [style=whitenode] (14) at (2.25, 5.5) {};
		\node [style=whitenode] (15) at (0, 3) {};
		\node [style=whitenode] (16) at (2.25, 0.5) {};
		\node [style=whitenode] (17) at (4.5, 3) {};
		\node [style=rednode] (23) at (-3.25, 3) {};
		\node [style=whitenode] (24) at (10, 5.5) {};
		\node [style=whitenode] (25) at (7.75, 3) {};
		\node [style=whitenode] (26) at (10, 0.5) {};
		\node [style=whitenode] (27) at (12.25, 3) {};
		\node [style=whitenode] (28) at (14.5, 5.5) {};
		\node [style=whitenode] (29) at (12.25, 3) {};
		\node [style=whitenode] (30) at (14.5, 0.5) {};
		\node [style=whitenode] (31) at (16.75, 3) {};
		\node [style=rednode] (37) at (11, 3) {};
		\node [style=none] (38) at (0, -1.5) {(b)};
		\node [style=none] (39) at (12.25, -1.5) {(c)};
		\node [style=whitenode] (40) at (18, -7) {};
		\node [style=whitenode] (41) at (15.5, -4.75) {};
		\node [style=whitenode] (42) at (13, -7) {};
		\node [style=whitenode] (43) at (15.5, -9.25) {};
		\node [style=whitenode] (44) at (23, -7) {};
		\node [style=whitenode] (45) at (20.5, -4.75) {};
		\node [style=whitenode] (46) at (18, -7) {};
		\node [style=whitenode] (47) at (20.5, -9.25) {};
		\node [style=whitenode] (52) at (-18.75, -4.5) {};
		\node [style=whitenode] (53) at (-21, -7) {};
		\node [style=whitenode] (54) at (-18.75, -9.5) {};
		\node [style=whitenode] (55) at (-16.5, -7) {};
		\node [style=whitenode] (56) at (-14.25, -4.5) {};
		\node [style=whitenode] (57) at (-16.5, -7) {};
		\node [style=whitenode] (58) at (-14.25, -9.5) {};
		\node [style=whitenode] (59) at (-12, -7) {};
		\node [style=whitenode] (64) at (-6.75, -4.5) {};
		\node [style=whitenode] (65) at (-9, -7) {};
		\node [style=whitenode] (66) at (-6.75, -9.5) {};
		\node [style=whitenode] (67) at (-4.5, -7) {};
		\node [style=whitenode] (68) at (-2.25, -4.5) {};
		\node [style=whitenode] (69) at (-4.5, -7) {};
		\node [style=whitenode] (70) at (-2.25, -9.5) {};
		\node [style=whitenode] (71) at (0, -7) {};
		\node [style=whitenode] (76) at (4.5, -4.5) {};
		\node [style=whitenode] (77) at (2, -7) {};
		\node [style=whitenode] (78) at (4.5, -9.5) {};
		\node [style=whitenode] (79) at (6.75, -7) {};
		\node [style=whitenode] (80) at (9, -4.5) {};
		\node [style=whitenode] (81) at (6.75, -7) {};
		\node [style=whitenode] (82) at (9, -9.5) {};
		\node [style=whitenode] (83) at (11.25, -7) {};
		\node [style=none] (88) at (-16.5, -11.5) {(d)};
		\node [style=none] (89) at (-16.5, -10.75) {};
		\node [style=none] (90) at (-4.5, -11.5) {(e)};
		\node [style=none] (91) at (7, -11.5) {(f)};
		\node [style=none] (92) at (18.5, -11.5) {(g)};
		\node [style=rednode] (93) at (-18.75, -6.25) {};
		\node [style=rednode] (94) at (-3, -7) {};
		\node [style=rednode] (95) at (10, -7) {};
		\node [style=rednode] (96) at (15.5, -6.25) {};
	\end{pgfonlayer}
	\begin{pgfonlayer}{edgelayer}
		\draw [style={shadow_}] (0.center)
			 to (1.center)
			 to (2.center)
			 to (3.center)
			 to cycle;
		\draw [style={shadow_pure}] (4.center)
			 to (5.center)
			 to (6.center)
			 to (7.center)
			 to cycle;
		\draw [style={black_thick}] (0) to (1);
		\draw [style={black_thick}] (2) to (1);
		\draw [style={black_thick}] (2) to (3);
		\draw [style={black_thick}] (3) to (0);
		\draw [style={black_thick}] (0) to (2);
		\draw [style={black_thick}] (5) to (4);
		\draw [style={black_thick}] (4) to (7);
		\draw [style={black_thick}] (7) to (6);
		\draw [style={black_thick}] (6) to (5);
		\draw [style={black_thick}] (4) to (6);
		\draw [style=rededge] (0) to (9);
		\draw [style=rededge] (9) to (1);
		\draw [style=rededge] (9) to (2);
		\draw [style=rededge] (9) to (3);
		\draw [style={shadow_}] (10.center)
			 to (11.center)
			 to (12.center)
			 to (13.center)
			 to cycle;
		\draw [style={shadow_pure}] (14.center)
			 to (15.center)
			 to (16.center)
			 to (17.center)
			 to cycle;
		\draw [style={black_thick}] (10) to (11);
		\draw [style={black_thick}] (12) to (11);
		\draw [style={black_thick}] (12) to (13);
		\draw [style={black_thick}] (13) to (10);
		\draw [style={black_thick}] (10) to (12);
		\draw [style={black_thick}] (15) to (14);
		\draw [style={black_thick}] (14) to (17);
		\draw [style={black_thick}] (17) to (16);
		\draw [style={black_thick}] (16) to (15);
		\draw [style={black_thick}] (14) to (16);
		\draw [style=rededge] (10) to (23);
		\draw [style=rededge] (23) to (11);
		\draw [style=rededge] (23) to (12);
		\draw [style=rededge] (23) to (15);
		\draw [style={shadow_}] (24.center)
			 to (25.center)
			 to (26.center)
			 to (27.center)
			 to cycle;
		\draw [style={shadow_pure}] (28.center)
			 to (29.center)
			 to (30.center)
			 to (31.center)
			 to cycle;
		\draw [style={black_thick}] (24) to (25);
		\draw [style={black_thick}] (26) to (25);
		\draw [style={black_thick}] (26) to (27);
		\draw [style={black_thick}] (27) to (24);
		\draw [style={black_thick}] (24) to (26);
		\draw [style={black_thick}] (29) to (28);
		\draw [style={black_thick}] (28) to (31);
		\draw [style={black_thick}] (31) to (30);
		\draw [style={black_thick}] (30) to (29);
		\draw [style={black_thick}] (28) to (30);
		\draw [style=rededge] (24) to (37);
		\draw [style=rededge] (37) to (26);
		\draw [style=rededge] (25) to (37);
		\draw [style=rededge] (37) to (29);
		\draw [style={shadow_}] (40.center)
			 to (41.center)
			 to (42.center)
			 to (43.center)
			 to cycle;
		\draw [style={shadow_pure}] (44.center)
			 to (45.center)
			 to (46.center)
			 to (47.center)
			 to cycle;
		\draw [style={black_thick}] (40) to (41);
		\draw [style={black_thick}] (42) to (41);
		\draw [style={black_thick}] (42) to (43);
		\draw [style={black_thick}] (43) to (40);
		\draw [style={black_thick}] (40) to (42);
		\draw [style={black_thick}] (45) to (44);
		\draw [style={black_thick}] (44) to (47);
		\draw [style={black_thick}] (47) to (46);
		\draw [style={black_thick}] (46) to (45);
		\draw [style={black_thick}] (44) to (46);
		\draw [style={shadow_}] (52.center)
			 to (53.center)
			 to (54.center)
			 to (55.center)
			 to cycle;
		\draw [style={shadow_pure}] (56) to (57);
		\draw [style={shadow_pure}] (57) to (58);
		\draw [style={shadow_pure}, in=225, out=45] (58) to (59);
		\draw [style={shadow_pure}] (59) to (56);
		\draw [style={black_thick}] (52) to (53);
		\draw [style={black_thick}] (54) to (53);
		\draw [style={black_thick}] (54) to (55);
		\draw [style={black_thick}] (55) to (52);
		\draw [style={shadow_pure}] (59.center)
			 to (56.center)
			 to (57.center)
			 to (58.center)
			 to cycle;
		\draw [style={black_thick}] (56) to (57);
		\draw [style={black_thick}] (57) to (58);
		\draw [style={black_thick}] (58) to (59);
		\draw [style={black_thick}] (59) to (56);
		\draw [style={shadow_}] (64.center)
			 to (65.center)
			 to (66.center)
			 to (67.center)
			 to cycle;
		\draw [style={shadow_pure}] (68) to (69);
		\draw [style={shadow_pure}] (69) to (70);
		\draw [style={shadow_pure}, in=225, out=45] (70) to (71);
		\draw [style={shadow_pure}] (71) to (68);
		\draw [style={black_thick}] (64) to (65);
		\draw [style={black_thick}] (66) to (65);
		\draw [style={black_thick}] (66) to (67);
		\draw [style={black_thick}] (67) to (64);
		\draw [style={shadow_pure}] (71.center)
			 to (68.center)
			 to (69.center)
			 to (70.center)
			 to cycle;
		\draw [style={black_thick}] (68) to (69);
		\draw [style={black_thick}] (69) to (70);
		\draw [style={black_thick}] (70) to (71);
		\draw [style={black_thick}] (71) to (68);
		\draw [style={shadow_}] (76.center)
			 to (77.center)
			 to (78.center)
			 to (79.center)
			 to cycle;
		\draw [style={shadow_pure}] (80) to (81);
		\draw [style={shadow_pure}] (81) to (82);
		\draw [style={shadow_pure}, in=225, out=45] (82) to (83);
		\draw [style={shadow_pure}] (83) to (80);
		\draw [style={black_thick}] (76) to (77);
		\draw [style={black_thick}] (78) to (77);
		\draw [style={black_thick}] (78) to (79);
		\draw [style={black_thick}] (79) to (76);
		\draw [style={shadow_pure}] (83.center)
			 to (80.center)
			 to (81.center)
			 to (82.center)
			 to cycle;
		\draw [style={black_thick}] (80) to (81);
		\draw [style={black_thick}] (81) to (82);
		\draw [style={black_thick}] (82) to (83);
		\draw [style={black_thick}] (83) to (80);
		\draw [style={black_thick}] (53) to (57);
		\draw [style={black_thick}] (56) to (58);
		\draw [style={black_thick}] (65) to (69);
		\draw [style={black_thick}] (68) to (70);
		\draw [style={black_thick}] (77) to (81);
		\draw [style={black_thick}] (80) to (82);
		\draw [style=rededge] (52) to (93);
		\draw [style=rededge] (93) to (53);
		\draw [style=rededge] (93) to (57);
		\draw [style=rededge] (93) to (54);
		\draw [style=rededge] (94) to (68);
		\draw [style=rededge] (69) to (94);
		\draw [style=rededge] (94) to (70);
		\draw [style=rededge] (94) to (71);
		\draw [style=rededge] (80) to (95);
		\draw [style=rededge] (95) to (82);
		\draw [style=rededge] (81) to (95);
		\draw [style=rededge] (95) to (83);
		\draw [style=rededge] (41) to (96);
		\draw [style=rededge] (96) to (42);
		\draw [style=rededge] (96) to (46);
		\draw [style=rededge] (96) to (43);
	\end{pgfonlayer}
\end{tikzpicture}

	\caption{$D$ is obtained from $D'$ 
		by a $4$-join operation}     
	\label{uc-faces}
\end{figure}
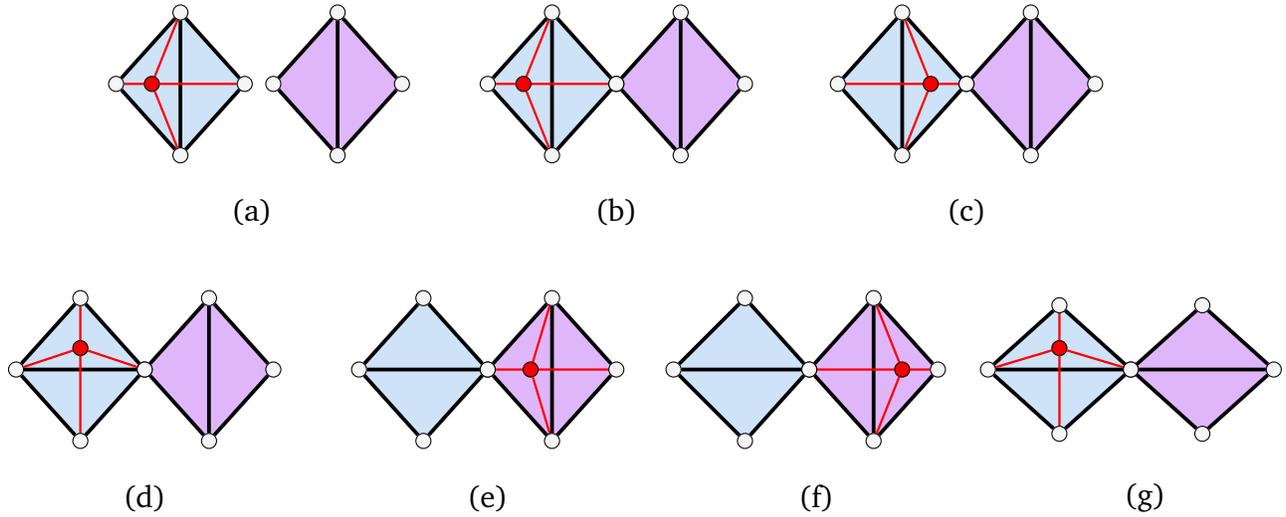

Thus, Claim~\ref{nc0} holds.
By Claims~\ref{nc-1} and~\ref{nc0}, 
every $1$-planar drawing of a $1$-planar $4$-tree of order at least $8$ belongs to $\Phi$.
The result holds.
\end{proof}

\section{Proof of Theorem \ref{4_conn}}

We first provides a sufficient condition for a graph to be 
Hamiltonian-connected. 

\begin{proposition}
	\label{Ham-con}
	Let $G$ be a graph with 
	two vertices  $a$ and $u$ 
	such that 
	$d_G(u)\ge 3$, $N_G[u]\subseteq N_G[a]$ and 
	$|N_G(a)\setminus N_G[u]|\le 1$.
	If both $G-u$ and $G-\{u,a\}$ 
	are Hamiltonian-connected,
	then $G$ is also Hamiltonian-connected.
\end{proposition}

\begin{proof}
	By the given condition, we may
	assume that $\{a,b,c\}\subseteq N_G(u)$.
	Since $N_G[u]\subseteq N_G[a]$, 
	$\{u, b,c\}\subseteq N_G(a)$,
	as shown in Fig.~\ref{HC-1}. 
	Let $G'=G-u$ and $G''=G-\{u,a\}$.
	We only need to prove that 
	for 	any two vertices $x$ and $y$
	in $G$, 
	there is a Hamiltonian path in $G$ 
	connecting $x$ and $y$.
	By the given conditions,  
	$G'$ has a Hamiltonian path $P$
	connecting $b$ and $y$. 
	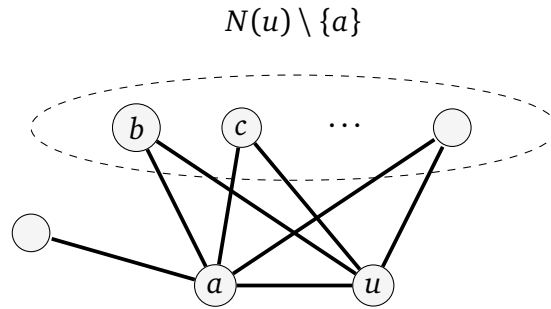
\begin{figure}[h!]
		\centering 
	\begin{tikzpicture}[scale=0.7]
	\begin{pgfonlayer}{nodelayer}
\tikzset{whitenode/.style={fill={rgb,255: red,245; green,245; blue,245}, draw=black, shape=circle, minimum size=0.1cm, inner sep=2.5pt,scale=1.0}}
\tikzset{whitenode1/.style={fill={rgb,255: red,245; green,245; blue,245}, draw=black, shape=circle, minimum size=0.1cm, inner sep=0.3pt,scale=5}}
\draw[dashed] (-2,3) ellipse (5cm and 1cm);
		\node [style=whitenode] (0) at (-5, 3) {$b$};
		\node [style=whitenode] (1) at (-3, 3) {$c$};
		\node [style=whitenode1] (2) at (1, 3) {};
		\node [style=none] (3) at (-1, 3) {$\cdots$};
		\node [style=whitenode] (4) at (-3.5, 0) {$a$};
		\node [style=whitenode] (5) at (-0.5, 0) {$u$};
		\node [style=whitenode1] (6) at (-7, 1) {};
		\node [style=none] (7) at (-2, 5) {$N(u)\setminus \{a\}$};
	\end{pgfonlayer}
	\begin{pgfonlayer}{edgelayer}
		\draw [style={black_thick}] (4) to (5);
		\draw [style={black_thick}] (5) to (1);
		\draw [style={black_thick}] (5) to (0);
		\draw [style={black_thick}] (4) to (0);
		\draw [style={black_thick}] (4) to (1);
		\draw [style={black_thick}] (4) to (2);
		\draw [style={black_thick}] (2) to (5);
		\draw [style={black_thick}] (6) to (4);
	\end{pgfonlayer}
\end{tikzpicture}

		\caption{$\{a,b,c\}\subseteq N_G(u)\subseteq N_G[a]$ and $|N_G(a)\setminus N_G[u]|\le 1$
		}     
		\label{HC-1}
	\end{figure}

	We divide the proof into three cases.
	
	\begin{case}
		$u\in \{x,y\}$.
	\end{case}
	
	By symmetry, assume that $x=u$. Then $y\in V(G')$. 
	Obviously, $y\ne a$ or $y\ne b$. 
	Assume that $y\ne b$.
	Combining $P$ and edge $ub$ produces a  Hamiltonian path in $G$ between $u$ (i.e., $x$) and $y$.

	\begin{case}
		$u\notin \{x,y\}$ and $a\notin \{x,y\}$. 
	\end{case}
	
	Since $|N_G(a)\setminus N_G[u]|\le 1$,  at least one edge $e^*$ in the set $\{av: v\in N_G(u)\}$ is on $P$.
	Without loss of generality, 
	assume that $e^* =ab$. 
	Let $P'$ be the path in $G$ 
	obtained from $P$ by replacing 
	edge $ab$ by the path $aub$. 
	Clearly, 
	$P'$ is  a Hamiltonian path in $G$
	connecting $x$ and $y$.
	
	\begin{case}
		$u\notin \{x,y\}$ and $a\in \{x,y\}$. 
	\end{case}
	
	By symmetry, assume that $a=x$.
	By the given condition,
	$G''$ is Hamiltonian-connected. 
	Then, 
	$G''$ has Hamiltonian paths
	$P_1$ connecting $b$ and $y$
	when $b\ne y$,
	and 
	$P_2$ connecting $c$ and $y$
	when $c\ne y$.
	If  $b\ne y$,  then 
	combining path $P_1$ 
	and path $aub$ yields a 
	Hamiltonian path  in $G$ 
	connecting $a$ (i.e., $x$) and $y$. 
	If $b=y$, 
	combining path $P_2$ 
	and path $auc$ yields a 
	Hamiltonian path  in $G$ 
	connecting $a$ (i.e., $x$) and $y$. 
	
	Hence the result holds.	
\end{proof}

By the definition of $k$-trees,
some conclusions on simplicial vertices of 
$k$-trees can be obtained.

\begin{lemma}\label{Chord} 
	Let $G$ be a $k$-tree of order $n$.
	If $n\ge k+2$, then $G$ has the following properties:
	\begin{enumerate}
		\item $G$ has at least two simplicial vertices; and 
		\item any two simplicial vertices in $G$ are not adjacent, and 
		\item for any 	simplicial vertex $v$ of $G$, 	if $n\ge k+3$, 
		then 	$G-v$ does not have more 	simplcial vertices than $G$ has.
	\end{enumerate} 
\end{lemma}

\begin{proof}
The result is obvious when $n=k+2$. 
For $n>k+2$, let $v$ be a simplicial vertex. Then $G-v$ is a $k$-tree 
of order $n-1\ge k+2$. 
By induction,  $G-v$ has all these three properties. 

Let $S(G)$ denote  
the set of simplicial vertices in $G$.
If no vertex in $N_G(v)$ belongs to 
$S(G-v)$, 
then 
$S(G)=S(G-v)\cup \{v\}$.
If some vertex $u\in N_G(v)$ belongs to 
$S(G-v)$, then $S(G)=
(S(G-v)\setminus \{u\})\cup \{v\}$.

Therefore $G$ also has these three properties. 
\end{proof}

Applying Proposition~\ref{Ham-con},
we can obtain the following 
conclusion on $k$-trees.

\begin{proposition}
	\label{ktree-HC}
	For any $k$-tree $G$, where $k\ge 3$, 
	if $|V(G)|\ge k+2$ 
	and $G$ has only two simplicial vertices, then 
	$G$ is Hamiltonian-connected.
\end{proposition}

\begin{proof}
	Let $G$ be a $k$-trees.
	Then $|V(G)|\ge k+1$.
	If $k\le k+1$, then 
	$G$ is a complete graph
	and thus it is Hamiltonian-connected.  
		If $|V(G)|=k+2$, then $G=K_{k+2}-e$,
	which is also Hamiltonian-connected.
	
	Now assume 
	the result holds for all $k$-trees
	of order less than $n$, where 
	$n\ge k+3$.
	Let $G$ be any $k$-tree of order $n$.
	By Lemma~\ref{Chord} (i), $G$ has at least two non-adjacent simplicial vertices. By the given conditions,
	$G$ has exactly two simplicial vertices, say $u$ and $v$.
	
	By the definition of $k$-tree, $G':=G-u$ is a $k$-tree of order $n-1$ 
	($\ge k+2$). By Lemma~\ref{Chord} (iii), $G'$ has exactly two simplcial vertices and 
	one of them must be $v$. 
	Let $a$ be another simplicial vertex of $G'$. 
	Clearly, $a\in N_G(u)$. 
	By the inductive assumption, 
	both $G-u$ and 	$G'-a$ ($=G-\{u,a\}$) are Hamiltonian connected. 
	
	Since $ua\in E(G)$, $u$ is a simplicial 
	vertex of $G$ with $d_G(u)=k\ge 3$
	and $a$ is a simplicial 
	vertex of $G'$ with $d_{G'}(a)=k$,
	we have 
	$
	N_G[u]\subseteq N_G[a]$ and 
	$
	|N_G(a)\setminus N_G[u]|=1.
	$
	By Proposition~\ref{Ham-con},
	$G$ is Hamiltonian-connected. 
	
	Hence the result holds.
\end{proof}

 Now we are going to prove 
 Theorem \ref{4_conn}.

 \emph{Proof of Theorem \ref{4_conn}}: 
 Let $G$ be a $1$-planar graph which is $4$-connected and chordal.
 
 Let $n$ be the order of $G$. 
 Then $n\ge 5$. 
 If $n=5$, then $G\cong K_5$.
 If $n=6$, then $G\cong K_6-e$ 
 or $K_6$.
 Clearly, $G$ is Hamiltonian-connected
 when $n\in  \{5,6\}$.
 Now assume that $n\ge 7$. 
 By Proposition~\ref{main0},
 $G$ is a $4$-tree. 
 In the following, we first show that 
 $G$ has exactly two simplicial vertices.

If $n=7$, by Proposition~\ref{4t-order7},
$G$ has a $1$-planar drawing
isomorphic to $B_1, B_2$ or $B_3$
shown in Fig.~\ref{4t-or7},
implying that $G$ is a $4$-tree with exactly two simplicail vertices. 
 
 When $n\ge 8$, 
 by Lemma~\ref{set-phi} 
 and Proposition~\ref{main-4t},
 $G$ has exactly two two simplicail vertices. 
 
 The result then follows from Proposition~\ref{ktree-HC}. 
 \proofend

\section{Unsolved problems}

 The toughness of a graph is  closely associated with Hamiltonicity. 
The \textit{toughness} of a graph $G$,
denoted by $\tau(G)$, 
is the minimum value of $\frac{|X|}{c(G-X)}$ over all  
non-empty subsets $X$ of $V(G)$
with $c(G-X)>1$, where  
$c(H)$ is the number of components of a graph $H$.
The toughness of a complete graph is defined as being infinite. 
We say that a graph is \textit{$t$-tough} if its toughness is at least $t$. 
There are some known results on 
the Hamiltonicity of chordal graphs 
in terms of their toughness. 
For example, every 10-tough chordal graph is Hamiltonian \cite{Kabela},
and  every %\red{more than 1-tough}
chordal planar graph of order 
at least three and toughness greater than one is Hamiltonian \cite{Bohme}.

% Every \red{more than 1-tough} chordal planar graph \red{with at least three vertices of toughness greater than one} is Hamiltonian \cite{Bohme}.

Note that a $t$-tough graph is always $\lceil 2t \rceil $-vertex-connected \cite{Chvatal}. So the corollary below follows directly from Theorem \ref{4_conn}.

\begin{corollary}\label{corol1}
	Every $1$-planar chordal graph 
	$G$ with $\tau(G)>\frac{3}{2}$
	is  Hamiltonian-connected.
\end{corollary}

 Chv\'{a}tal (1973) \cite{Chvatal} conjectured that all graphs which are more than $\frac{3}{2}$-tough are Hamiltonian, but this was disproved by Bauer et al. (2000), who showed that not every 2-tough graph is Hamiltonian \cite{Bauer}. Chvátal's toughness conjecture posits that there exists a toughness threshold $t_0$ above which $t_0$-tough graphs are always Hamiltonian; its truth remains unresolved.
  
  Corollary \ref{corol1}  states that every 1-planar chordal graph with toughness greater than $\frac{3}{2}$ is Hamiltonian-connected.  Naturally, we pose the following problem:

\begin{problem}\label{pp1}
Is there a  $1$-planar chordal graph with toughness   $\frac{3}{2}$ non-Hamiltonian?
\end{problem}

Notice that neither the examples in \cite{Nishizeki} nor our examples in Remark \ref{r1} can be used for Problem \ref{pp1}. This is because the toughness of every chordal  planar non-Hamiltonian graph in \cite{Nishizeki} is always 1, while our examples can be shown to be at most $\frac{5}{6}$.

\section{Acknowledgment}

 The work is supported by the National Natural 
	Science Foundation of China (Grant No. 
	12271157, 12371340, 12371346), the Natural Science Foundation of Hunan Province, China (Grant 
	No. 2022JJ30028),  Hunan Provincial Department of Education(No. 21A0590) and Natural Science Foundation of Changsha (No. kq2208001).

We claim that there is no conflict of interest in our paper.
No data was used for the research described in the article.

 \nocite{*}
%\bibliographystyle{bib}
%\bibliography{pap.bib}
\end{document}